\documentclass[a4paper,11pt]{amsart}
\usepackage{amsmath,amssymb,colordvi}

\newtheorem{theorem}{Theorem}
\newtheorem{corollary}[theorem]{Corollary}
\newtheorem{lemma}[theorem]{Lemma}
\newtheorem{example}[theorem]{\it Example}
\newtheorem{proposition}[theorem]{Proposition}
\newtheorem{definition}[theorem]{Definition}

\newtheorem{remark}[theorem]{\it Remark}

\font\tenBb=msbm10 \font\sevenBb=msbm7 \font\fiveBb=msbm5
\newfam\Bbfam \textfont\Bbfam=\tenBb \scriptfont\Bbfam=\sevenBb
\scriptscriptfont\Bbfam=\fiveBb

\def\Bb{\fam\Bbfam\tenBb}
\def\R{{\Bb R}}
\def\C{{\Bb C}}
\def\N{{\Bb N}}

\def\Dk{{\mathcal D}_k}
\def\B{{\mathcal B}}

\begin{document}
\title[Sharp extensions for  solutions of abstract Cauchy  problems]{Sharp extensions for convoluted solutions of \\abstract Cauchy problems}

\author{Valentin  Keyantuo}
\address{University of Puerto Rico, Department of Mathematics,
Faculty of Natural Sciences, Box 70377,  San Juan PR 00936-8377 U.S.A.}
%\curraddr{}
\email{keyantuo@uprrp.edu}
\thanks{\noindent P.J. Miana and L. S\'{a}nchez-Lajusticia have been partly supported by supported by Project MTM2010-16679, DGI-FEDER, of the MCYTS; Project E-64, D.G. Arag\'on, and JIUZ-2012-CIE-12, Universidad de Zaragoza, Spain.
This paper was written during several visits of P. J. Miana to the Department of Mathematics, University of Puerto Rico, R\'io Piedras Camppus. He is grateful to the Department for its hospitality.}
%Projects MTM2007-61446, DGI-FEDER,  (MCYT) of Spain, and Project
%E-64, D.G. Arag\'on, Spain. }

\author{Pedro J. Miana}
\address{Universidad de Zaragoza, Departamento de Matem\'{a}ticas \& I.U.M.A.,
50.009, Zaragoza.}
 \email{pjmiana@unizar.es}

 \author{Luis S\'{a}nchez-Lajusticia}
\address{Universidad de Zaragoza, Departamento de Matem\'{a}ticas \& I.U.M.A.,
50.009, Zaragoza.}
 \email{luiss@unizar.es}
\thanks{}

\thanks{ }

% General info


\begin{abstract}
In this paper we give sharp extension results for convoluted solutions of abstract Cauchy problems in Banach spaces. The main technique is the use of algebraic structure (for usual convolution product $\ast$)  of these solutions which are defined by a version of the Duhamel formula. We define  algebra homomorphisms from a new class of test-functions and apply our results to concrete operators. Finally, we introduce the notion of $k$-distribution semigroups to extend previous concepts of distribution semigroups.

\end{abstract}
\maketitle

\section{introduction}
\setcounter{theorem}{0}

Let $A$ be a closed linear operator on a Banach space $X$ and $\tau>0.$ It is well known (see \cite[Theorem 3.1]{[vCa85]}, \cite{[ABHN]} or \cite[Theorem 2.1]{[AEK]}) that if for every $x\in X, $ the Cauchy problem
\begin{equation}\label{Cpvca}
\begin{cases}
\displaystyle{  u^\prime(t)}=A u(t)  +x,\, 0\le t<\tau\\
u(0)=0\\
 \end{cases}
\end{equation}
has a unique solution $\displaystyle u\in C^1([0,\, \tau), X)\cap C([0,\, \tau), D(A))$ (where $D(A)$ is endowed with the graph norm), then $A$ is the generator of a strongly continuous semigroup.\\ This means that the solutions, initially obtained  on $[0,\, \tau),$ admit extensions to $[0,\, \infty)$ without loss of regularity. Moreover, these solutions are (uniformly) exponentially bounded in the sense that there exist $M>0, \, \omega\in \mathbb{R}$ independent of $u(.)$ such that all solutions $u(.)$ satisfy $\Vert u(t)\Vert\le Me^{\omega t},\, t\ge 0.$   Also note that in this case, $A$ is necessarily densely defined.

 In 1833, J.M.C. Duhamel considered the following evolution problem corresponding to the initial-boundary value problem for the heat equation in a domain $\Omega$ ($\Omega$ is an open subset  of $\R^n$):
\begin{equation}\label{equa}
\begin{cases}
\displaystyle{\partial u\over \partial t}=\Delta u, &(t,x)\in \R^+ \times \Omega\\
u(0,x)= u_0(x), & x\in \Omega;\\
u(t, \cdot)\vert_{\partial \Omega}= g(t, \cdot), & t>0.\\
\end{cases}
\end{equation}
and proposed the following formula to express the solution of (\ref{equa}).
$$
u(t,x)=\int_0^t{\partial \over \partial t}u(\lambda, t-\lambda,x)d\lambda, \quad t>0,
$$
where $u(\lambda,t,x)$ is a  solution of  (\ref{equa}) for a particular function $g(\cdot, \lambda_0)$ fixed $\lambda_0$ (\cite{[Du]}). This formula allows one to reduce the Cauchy problem for an in-homogeneous partial differential equation to the Cauchy problem for the
corresponding homogeneous equation. This formula (known also as Duhamel's principle) is of widespread use  in partial differential equations and has been studied in  a large number of papers, see for example \cite{DL, Sa}. In the \cite{Sa}, the author extends the Duhamel principle to fractional order equations.

Let $k:(0,T) \to \C$  be a locally integrable function, $X$  a Banach space and $x\in X.$ The Cauchy problem

\begin{equation}\label{convsem-0}
\begin{cases}
\displaystyle{v'(t) = Av(t) + K(t)x} \quad &0<t<T,\\
 v(0) = 0,
\end{cases}
\end{equation}

is   called $K$-convoluted Cauchy problem where $K(t)=\int_0^tk(s)xds$ for $0<t<T$.
If there exists a solution of the abstract
initial value problem $u'(t) = Au(t)$ for $ 0<t < T,$ $u(0) = x$
then, as usual for a nonhomogeneous equation,  we have $v=u\ast K$ ($\ast$ is the usual convolution in $\R^+$).
Local $k$-convoluted semigroups are defined using a version of Duhamel's formula and were introduced in \cite{C, CL}. This class of semigroups includes $C_0$-semigroups and integrated semigroups as particular examples, see a complete treatment in \cite[Section 1.3.1]{[MF]}, \cite[Chapter 2]{ko} and  other details in \cite{KLM,  [KP]}. The concept of regularized semigroups is covered by taking $k(t)  \equiv C,\,\, 0\le t<\tau,$ where
$C$  is a bounded and injective operator on $X.$

 Contrary to what happens in the case of equation (\ref{Cpvca}), if $k:\, (0,\,\,\infty)\longrightarrow \C$ is locally integrable and for every $x\in X$ there exists a unique solution $\displaystyle u\in C^1([0,\, \tau), X)\cap C([0,\, \tau), D(A))$ for (\ref{convsem-0}), it is generally not the case that these solutions can be extended to $[0,\,\infty)$, nor that exponential boundedness is achieved in case one can extend the solutions. In this case, we say that $A$ is the generator of a local $k-$convoluted semigroup.  However, there is an underlying  algebraic structure of $k$-convoluted semigroups which leads to the following  the extension property: the solution of  $k$-convoluted problem on $[0, T)$ is used to express the solution of the $k\ast k$-convoluted problem on $[0,2T)$, see \cite[Section 2]{CL} and \cite[Theorem 2.1.1.9]{ko}. Stated otherwise, when (\ref{convsem-0}) is well posed on $[0,\,\tau)$, the equation in which we replace $k(.)$ with $(k\ast k)(.)$ is well posed on $[0,\,2\tau).$ Our result (Theorem \ref{ext}) provide a sharpening of this extension property.

The special case of $k(t)={t^{\alpha-1}\over \Gamma(\alpha)}$ with $\alpha >0$ defines the $\alpha$-times integrated semigroup. Originally they were the first example of convoluted semigroups. An extension formula for $n$-times integrated semigroups (for $n\in \N$) was given in \cite[Section IV, (4.2)]{[AEK]} and for $\alpha$-times integrated semigroup in \cite[Formula (5)]{M1} with $\alpha >0$. Extensions of local $\alpha$-times integrated $C$-semigroups were given in \cite[Theorem 6.1]{LS} and automatic extension of local regularized semigroups appears in \cite[Section 2]{[WG]}.

 The main objective of this paper is to illustrate  the algebraic structure of  local $k$-convoluted semigroups.  In \cite[Section 5]{KLM}, authors consider global exponentially bounded convoluted semigroups and algebra homomorphisms defined via these classes of semigroups; in fact both concepts are equivalent, see \cite[Theorem 5.7]{KLM}.

 In the context of  local convoluted semigroups as well as global non-exponentially bounded convoluted semigroups, this point of view is not so evident. This is due to the fact that the Laplace transform is an essential tool in the global exponential case. First we need some technical identities which involve convolution products in Section 2. In Section 3, we introduce a new test function space, ${\mathcal D}_{k^{\ast\infty}}$(in Definition \ref{definf}) which will play a fundamental role in this paper, see Theorem \ref{main2} and Appendix. In the fourth section, we give one of the main results of this paper. We derive a sharp extension theorem for local convoluted  semigroups (Theorem \ref{ext}). In Section 5, we use the extension formula to define algebra homomorphism from the test function space  ${\mathcal D}_{k^{\ast\infty}}$ via local convoluted semigroup (Theorem \ref{main2}). In Section 6, we apply our results to four concrete operators which generate (local and global) convoluted semigroups. In the global case and under the exponential boundedness assumption, the Laplace transform is used as a crucial tool. This is no longer the case when one consider the local case or the global one without the assumption of  exponential boundedness. In the case we consider in this paper,  algebras concerned are no longer Banach algebras but only locally convex algebras.

Historically distribution semigroups were introduced by J.L. Lions in \cite{[Li60]} in the early sixties in the exponential case with the Laplace transform of distributions as an important tool. The paper \cite{Ch71}  by  J. Chazarain presents an extension to the non exponential case and goes further to introduce the ultradistributional framework (see also the monograph \cite{[LM73]}). This class of vector-valued distribution (with a suitable algebraic structure) gives a equivalent approach to local integrated semigroups as was proven in  \cite[Theorem 7.2]{[AEK]}. For local convoluted semigroups, we present a similar approach in Appendix where we introduce $k$-distributions semigroups and we present their connections with local convoluted semigroups. The interest in the local case stems from the fact that for the general classes of generalized semigroups that have been introduced following Lions'  paper, by using the local approach, one is able to obtain a Banach space valued formulation that captures almost all the situations involved. The monographs \cite{ko}, \cite{[MF]} and the  references cited therein contain more information on distribution as well as ultradistribution semigroups. They also explore the ways in which they relate to local convoluted semigroups.

A similar and independent approach may be followed in the abstract Cauchy problem of second order or wave problem. In this case we need to consider local convoluted cosine functions and distribution cosine function (mainly  algebra homomorphism for cosine convolution product) see more details in \cite{MP}.

\section{Some identities for convolution products}
\setcounter{theorem}{0}

Let $L^1_{loc}(\R^+)$ the space of complex valued locally integrable functions on $\R^+$ and we consider the usual convolution product $\ast$, given by
$$\displaystyle (f\ast g)(t)=\int_0^tf(t-s)g(s)ds, \, t\ge0,$$
where $f, \, g\in L^1_{loc}(\R^+)$.  We write $k^{*2}$ instead of $k\ast k$ and then $k^{*n}=k*\left(k^{\ast(n-1)}\right)$ for $n> 2 $ is the $n-$fold convolution power of $k.$ The convolution product is associative and commutative. We also follow the notation $\circ$ to denote the dual convolution product of $\ast$ given by
$$\displaystyle (f\circ g)(t)=\int_t^\infty f(s-t)g(s)ds, \, t\ge0,$$
where $f, \, g\in L^1(\R^+)$. Note that
$$\max\{t\vert\, t\in {supp}(f\circ g)\}\le \max\{t\vert\, t\in {supp}( g)\}, \qquad f, \, g\in L^1(\R^+).$$

We denote by $\chi$ the constant function equal to $1$, i.e., $\chi(t)=1$ for $t\in \R^+$. This corresponds to the Heaviside function.
Moreover, for $\alpha>0,$  we set $j_\alpha(t)=\displaystyle \frac{t^{\alpha-1}}{\Gamma(\alpha)}, \,\, t\ge 0.$ It will be convenient to set $j_0=\delta_0,$ the Dirac measure concentrated at the origin. Observe that the following semigroup property holds: $\displaystyle j_\alpha\ast j_\beta=j_{\alpha+\beta},\,\, \alpha,\, \beta\ge 0.$  The following lemma will be used for the proof of the main result in Section 4.

\begin{lemma}\label{lemma} Take $0\le \tau \le t$ and $f,g\in L^1_{loc}(\R^+)$. Then
$$
\int_0^{t-\tau}f(t-s)\left(\chi\ast g\right)(s)ds+ \int_0^{\tau}g(t-s)\left(\chi\ast f\right)(s)ds=\left(g \ast \left(\chi\ast f\right)\right)(t)-  \left(\chi\ast g\right)(t-\tau)\left(\chi\ast f\right)(\tau).
$$
\end{lemma}
\begin{proof} Observe that $\displaystyle \frac{d}{ds}\int_s^t f(t-u)du=-f(t-s)$ and by simple change of variable we have:  $\displaystyle  \int_{t-\tau}^t f(t-u)du=\int_0^\tau f(s)ds$ and $\displaystyle  \int_s^t f(t-u)du=\int_0^{t-s}f(u)du.$  We integrate by parts in the following integral to obtain,
\begin{eqnarray*}
\int_0^{t-\tau}f(t-s)\int_0^sg(u)duds&=&-\int_{t-\tau}^tf(t-s)ds\int_0^{t-\tau}g(u)du+ \int_0^{t-\tau}g(s)\int_s^tf(t-u)duds\cr
&=&-\int_0^{t-\tau} g(u)du\int_0^{t-\tau}f(t-u)du+\int_0^{t-\tau}g(s)\int_0^s f(t-u)duds\cr
 &=&-\left(\chi\ast g\right)(t-\tau)\left(\chi\ast f\right)(\tau)+ \int_0^{t-\tau}g(s)\int_0^{t-s}f(u)duds, \cr
 \end{eqnarray*}
for $0\le \tau\le t$. Note that
\begin{eqnarray*}
 \int_0^{t-\tau}g(s)\int_0^{t-s}f(x)dxds&=& \int_0^{t}g(s)\int_0^{t-s}f(x)dxds- \int_{t-\tau}^tg(s)\int_0^{t-s}f(x)dxds\cr
 &=&\left(g\ast \left(\chi\ast f\right)\right)(t)-\int_{0}^\tau g(t-u)\int_0^{u}f(x)dxdu\cr
 &=&\left(g \ast \left(\chi\ast f\right)\right)(t)-\int_0^{\tau}g(t-u)\left(\chi\ast f\right)(u)du,
\end{eqnarray*}
and this  concludes the proof.
\end{proof}

Taking $f=j_\alpha $ and $g=j_\beta$ with $\alpha, \beta>0$ in Lemma \ref{lemma}, we get the equality
$$
\int_0^{t-\tau}{(t-s)^{\alpha-1}\over \Gamma(\alpha)}{s^{\beta}\over \Gamma(\beta+1)}ds+ \int_0^{\tau}{(t-s)^{\beta-1}\over \Gamma(\beta)}{s^{\alpha}\over \Gamma(\alpha+1)}ds={t^{\alpha+\beta}\over \Gamma(\alpha+\beta+1)}-  {(t-\tau)^{\beta}\over \Gamma(\beta+1)}{\tau^{\alpha}\over \Gamma(\alpha+1)}
$$
for $0\le \tau\le t$.

If we set $f=g$ in Lemma \ref{lemma}, we obtain:

\begin{corollary}\label{coro} Take $0\le \tau \le t$ and $f\in L^1_{loc}(\R^+)$. Then
\begin{equation}\label{equ}
\left(\int_0^{t-\tau}+ \int_0^{\tau}\right)f(t-s)\left(\chi\ast f\right)(s)ds=\left(f \ast \left(\chi\ast f\right)\right)(t)-  \left(\chi\ast f\right)(t-\tau)\left(\chi\ast f\right)(\tau).
\end{equation}
Further specializing to  $f=j_\alpha$ for $\alpha> 0$, yields the identity:
\begin{eqnarray*}
\left(\int_0^{t-\tau}+
\int_0^\tau\right) {(t-s)^{\alpha -1}\over \Gamma (\alpha
)}{s^{\alpha }\over \Gamma (\alpha+1 )}ds
={t^{2\alpha }\over \Gamma (2\alpha +1)}- {(t-\tau)^\alpha \over \Gamma (\alpha +1)}{\tau^\alpha \over
\Gamma (\alpha +1)}
\end{eqnarray*}
for $0\le \tau\le t$.
\end{corollary}

As a consequence of the last corollary, we obtain another proof of the following result given  in \cite[Lemma 2.1.12]{ko} for continuous functions
and in \cite[Lemma 3.1]{[KS]} for $f(t)=\displaystyle{t^{\alpha-1}\over \Gamma(\alpha)}$ and $\alpha>0$.
\begin{corollary} \label{jojo} For $f\in L^1_{loc}(\R^+)$ and   $s,  u\ge 0$  we have
\begin{eqnarray*}
\left(\int_0^{s+u}-\int_0^s-\int_0^u\right)f(u+s-r)f(r)dr=0;
\end{eqnarray*}
in particular for $\alpha >0$ and $s,  u\ge 0$, we get that
\begin{eqnarray*}
\left(\int_0^{s+u}-\int_0^s-\int_0^u\right)(u+s-r)^{\alpha-1}r^{\alpha-1}dr=0;
\end{eqnarray*}
\end{corollary}

\begin{proof} By change of variable, we write the identity (\ref{equ}) as
$$
\left(f \ast \left(\chi\ast
  f\right)\right)(t)-  \left(\chi\ast f\right)(t-\tau)\left(\chi\ast f\right)(\tau)= \left(\int_\tau^{t} +\int_{t-\tau}^t \right)
{f(x)}(\chi \ast f)(t-x)dx    $$
for $0\le \tau \le t$.

Now we observe that $$\displaystyle \frac{d}{dt}\left(f\ast(\chi\ast f)\right)(t)=(f\ast f)(t).$$ Similarly, we have
$$ \displaystyle \frac{d}{dt}\left((\chi\ast f)(t-\tau)(\chi\ast f)(\tau)\right)=f(t-\tau)(\chi\ast f)(\tau),$$
\begin{eqnarray*}
\frac{d}{dt}\int_t^\tau f(u)(\chi\ast f)(t-u)du&=&-f(t)(\chi\ast f)(0)-\int_\tau^t f(u)f(t-u)du\cr
&=&\int_t^\tau f(u)f(t-u)du,\cr
\end{eqnarray*}
and
\begin{eqnarray*}
\frac{d}{dt}\int_{t-\tau}^t f(u)(\chi\ast f)(t-u)du&=& f(t)(\chi\ast f)(0)+\int_0^t f(u)f(t-u)du\cr
&&-f(t-\tau)(\chi\ast f)(\tau)-\int_0^{t-\tau} f(u)f(t-u)du\cr
&=&\int_0^t f(u)f(t-u)du-f(t-\tau)(\chi\ast f)(\tau)\cr
&&-\int_0^{t-\tau} f(u)f(t-u)du\cr
\end{eqnarray*}
Differentiating with respect to the variable $t$ and using the above, we have:
\begin{eqnarray*}(f\ast f)(t)&=&
\int_\tau^{t}f(x)f(t-x)dx+
\int_{t-\tau}^t f(x)f(t-x)dx\cr
&=&
\int_0^{t-\tau}
f(t-s)f(s)ds+
\int_{0}^\tau f(t-s)f(s)ds\cr
\end{eqnarray*}
for $0\le \tau \le t$. Now take $t=s+u$, and $\tau=s$ and we conclude the proof.
\end{proof}

Take  $k\in L^1_{loc}([0,\tau))$, and we define $(k_t)_{t\in[0,\tau)}\subset L^1_{loc}([0,\tau))$ by
\begin{equation}\label{form}
k_t(s):= k(t-s)\chi_{[0, t]}(s), \qquad s\in [0,\tau).
 \end{equation}
A similar result was considered  in \cite[Proposition 2.2]{KLM} for functions belong to $L^1_{loc}(\R^+)$. Here we present a direct proof for  $ L^1_{loc}([0,\tau))$.

 \begin{theorem}\label{k-convo}Take  $k\in L^1_{loc}([0,\tau))$ and $(k_t)_{t\in[0,\tau)}$ defined by (\ref{form}). Then

 $$
k_t\ast k_s(x)= \int_t^{t+s}k(t+s-r)k_r(x)dr-\int_0^{s}k(t+s-r)k_r(x)dr, \qquad 0\le x<\tau,
$$
for  $0\le s,t\le t+s<\tau$.

 \end{theorem}

 \begin{proof} We consider (without loss of generality) that $0\le s\le t$. First we consider $0\le x\le s$. Then
$$
k_t\ast k_s(x)= \int_0^xk(t-(x-y))k(s-y)dy, \quad 0\le x<\tau,
$$
and
 \begin{eqnarray*}
 &\,&\left( \int_t^{t+s}-\int_0^{s}\right)k(t+s-r)k_r(x)dr= \left( \int_t^{t+s}-\int_x^{s}\right)k(t+s-r)k(r-x)dr\cr
 &\,& =\int_0^{s}k(u)k(t+s-u-x)du-\int_0^{s-x}k(t+s-x-y)k(y)dy\cr
 &\,&=\int_{s-x}^{s}k(u)k(t+s-u-x)du= \int_0^xk(s-y)k(t-x+y)dy
 \end{eqnarray*}
 where we have changed  variables  in several equalities. The other cases $s\le x\le t$, $t\le x\le t+s$ and $ t+s\le x<\tau$ are made following ideas. In particular we remark that $k_t\ast k_s(x)=0$ for $t+s\le x<\tau$.
 \end{proof}

\begin{remark} {\rm By Proposition \ref{equiv} and Theorem \ref{k-convo}, we may conclude that $(k_t)_{t\in[0,\tau)}$ is a local $k$-convoluted semigroup in $L^1_{loc}([0,\tau))$. In fact, note that $k_t= \delta_t\ast k$ where $(\delta_t)_{t\ge 0}$ is the Dirac measure concentrated at  $t= 0$. In this sense, $(k_t)_{t\in[0,\tau)}$ is the canonical local $k$-convoluted semigroup.
}
\end{remark}

\section{The Laplace transform and $k$-Test function spaces}\label{seccion}
\setcounter{theorem}{0}

Let $k\in L^{1}_{loc}(\mathbb{R}^{+})$.  We write by $\widehat k$ the usual Laplace transform of $k$, given by
$$
\widehat k(\lambda)= \lim_{N\to \infty}\int_0^N e^{-\lambda t} k(t)dt,
$$
in the case that there exists for $\lambda \in \C$; $\hbox{abs}(k)$ is defined by
$
\hbox{abs}(k):=\inf\{\Re \lambda; \hbox{ exist } \widehat k(\lambda)\},
$
see \cite[Section 1.4]{[ABHN]}. In the case that $\vert k(t)\vert \le Me^{\omega t}$ for a.e. $t\ge 0$ and $M, \omega>0$, we have that
$$
\widehat k(\lambda)= \int_0^\infty e^{-\lambda t} k(t)dt, \qquad \Re \lambda >\omega.
$$
  For $\lambda\in\C,$ we write $e_\lambda(t)=e^{-\lambda t},\, t\ge 0.$
\begin{lemma}\label{almo} Take $k\in L^{1}_{loc}(\mathbb{R}^{+})$ such that  $\vert k(t)\vert \le Me^{\omega t}$ for a.e. $t\ge 0$ and $M, \omega>0$. Then
$$
k\circ e_{-\lambda}=\widehat k(\lambda)e_{-\lambda}, \qquad \Re \lambda >\omega.
$$
In particular, if $\widehat k(\lambda)=0$ for some $\Re \lambda >\omega$, then $k\circ e_{-\lambda}=0$.
\end{lemma}

\begin{proof} Take $\lambda \in \C$ with $\Re \lambda >\omega$ and
$$
k\circ e_{-\lambda}(t)=\int_t^\infty k(s-t) e^{-\lambda s}ds=\int_0^\infty k(u) e^{-\lambda u}due^{-\lambda t}=\widehat k(\lambda)e_{-\lambda}(t)
$$
for $t>0$.
\end{proof}

We write by ${supp}(h)$ the usual support of a function $h$ defined in $\R$;  $\mathcal{D}$ is the space of $\mathcal{C}^{(\infty)}$ functions with compact support on $\R$ and $\mathcal{D}_0$ is the set of $\mathcal{C}^{(\infty)}$ functions with compact support on $[0, \infty)$, $\mathcal{D}_0\subset \mathcal{D}$. The space $\mathcal{D}$  will be equipped with  the Schwartz topology which turns it into a complete topological vector space. We denote the topology by  $\tau.$ In particular, sequential convergence in $\mathcal{D}$ is described by: let $(\phi_n)_{n\ge 1}\subset \mathcal{D}, \phi\in \mathcal{D}$, then $\phi_n\to_{\tau}\phi $ if and only if
\begin{itemize}
\item[(i)] there exists a compact subset $K\subset \R$ such that ${supp}(\phi_n), {supp}(\phi)\subset K$.
\item[(ii)] for any $j\ge 0$, $\phi^{(j)}_n\to \phi^{(j)}$ uniformly on compacts sets.
\end{itemize}
Note that $\mathcal{D}_0$ is a closed subspace of $\mathcal{D}$ and then $(\mathcal{D}_0, \tau)$ is a complete topological space (we keep the same notation for the topology $\tau$ and its restriction to the subspace $\mathcal{D}_0$).

We denote by $\mathcal{D}_+$ the set of functions defined by  $\phi_+:[0, \infty) \to \C$, given by $\phi_+(t):=\phi(t)$ for $t\ge 0$ and $\phi\in \mathcal{D}$ and define $\mathcal{K}: \mathcal{D}\to \mathcal{D}_+$ by $\mathcal{K}(\phi)=\phi_+$ for $\phi\in \mathcal{D}$.
Due to the extension theorem of R. T. Seeley \cite{[Se]}, there exists a linear continuous operator $\Lambda: \mathcal{D}_+\to \mathcal{D}$, such that $\mathcal{K}\Lambda = I_{\mathcal{D}_+}$; in particular if $\psi $ is a $\mathcal{C}^{(\infty)}$ functions on $[0,\infty)$ and  compact support  then $\psi \in \mathcal{D}_+$. The space  $\mathcal{D}_+$ is also  a complete topological vector spaces equipped with the $\tau$-topology of uniform convergence on compact subsets.

We define the operator $T'_{k}: \mathcal{D} \rightarrow \mathcal{D}$   by $f\mapsto T'_{k}(f):=k\circ f,$ that is,
 \begin{eqnarray*}
 T^\prime_k(f)(t)=\int_t^\infty k(s-t)f(s)ds,\, t\ge 0.
 \end{eqnarray*}
 We shall also use the same notation for the restriction to $\mathcal{D}_+$ that is, $T'_{k}: \mathcal{D}_+ \rightarrow \mathcal{D}_+$; however $T'_{k}:\mathcal{D}_0\not \rightarrow \mathcal{D}_0$. Note that $T'_k(f_u)= (T'_k(f))_u$,  where $f_u(t)=f(u+t)$ for $u,t\ge 0$ and $f\in \mathcal{D}_+ $.

In the case that $0\in {supp}(k)$, we have that $T'_{k}:\mathcal{D}_{+}\rightarrow\mathcal{D}_{+}$ is an injective, linear and continuous homomorphism such that
$$T'_{k}(f\circ g) = f \circ T'_{k}(g),\qquad f,g\in\mathcal{D}_{+},$$
see \cite[Theorem 2.5]{KLM}.
Then, we  define the space $\mathcal{D}_{k}$ by $\mathcal{D}_{k}:=T'_{k}(\mathcal{D}_{+})\subset \mathcal{D}_{+} $ and the right inverse map of  $T'_{k}$, i.e., $W_{k}:\mathcal{D}_{k} \rightarrow \mathcal{D}_{+}$ by
$$f(t)= T'_{k}(W_{k}(f))(t)=\int_{t}^{\infty}k(s-t)W_{k}f(s)ds,\qquad f\in\mathcal{D}_{k}, \quad t\geq0,$$
see \cite[Definition 2.7]{KLM}.  Note that the operator $W_{k}:\mathcal{D}_{+} \rightarrow \mathcal{D}_{+}$ is a closed operator $(D(W_k)=\mathcal{D}_{k})$, but we cannot apply the open mapping theorem to conclude that it is continuous.

It is clear that the subspace $\mathcal{D}_{k}$ is also a topological algebra: take $f, g \in\mathcal{D}_{k}$, then $f\ast g \in  \mathcal{D}_{k}$ (\cite[Theorem 2.10]{KLM}) and the map  $(f, g)\to f\ast g$ is continuous in $\mathcal{D}_{k}$. Moreover  $W_k(k\circ f)=f$  for $f \in \mathcal{D}_{+}$ and $f_u \in \mathcal{D}_{k}$, with
\begin{equation}\label{transla}
W_k(f_u)= (W_k(f))_u, \qquad f\in  \mathcal{D}_{k},\quad u\ge0.
\end{equation}

We have the following property to the effect that $W_k$ does not increase the support.\\
\begin{lemma} \label{equivalencia} Let $k\in L^{1}_{loc}(\mathbb{R}^{+})$ be such that  $0\in {supp}(k)$ and let $a>0$. Then $supp(f)\subset [0,a]$ if and only if $supp(W_kf)\subset [0,a]$ for $f\in\mathcal{D}_{k}$.
\end{lemma}

\begin{proof} Take $f\in\mathcal{D}_{k}$ such that $supp(f)\subset [0,a]$. Then
$$
0= f(t)=k\circ W_kf(t)= \int_t^\infty k(s-t)W_kf(s)ds, \qquad t\ge a,
$$
We write $t=a+r$ for $t\ge a$ and $r\ge 0$,
\begin{eqnarray*}
0= \int_t^\infty k(s-t)W_kf(s)ds= \int_0^\infty k(x)W_kf(x+a+r)dx= \int_r^\infty k(x-r)(W_kf)_a(x)dx,
\end{eqnarray*}
where $(W_kf)_a(x):=(W_kf)(x+a)$ for $x>0$. We apply the Titchmarsh-Foia\c s \cite[Theorem 2.4]{KLM} to conclude that $(W_kf)_a(x)=0$ for $x>0$, i.e., $supp(W_kf)\subset [0,a]$. Conversely, suppose that $supp(W_kf)\subset [0,a]$. It then follows from the representation
$$
f(t)=\int_t^\infty k(s-t)W_kf(s)ds, \qquad t\ge 0,
$$
that  $supp(f)\subset [0,a].$
\end{proof}

Note that in the case that $f\in \mathcal{D}_{k}$ then  $f^{(n)}\in \Dk$ and \begin{equation}\label{derivada}W_k(f^{(n)})=(W_kf)^{(n)}, \qquad n\ge 1;\end{equation}
  take $k,l \in L^1_{loc}(\R^+)$ such that $0\in
{supp}(k)\cap{supp}(l)$.  Then $0\in {supp}(k\ast l)$, ${\mathcal D}_{k\ast l}\subset
\Dk\cap {\mathcal D}_l$ and
 \begin{equation}\label{cop}
 W_kf=l\circ W_{k\ast l}f, \quad f\in {\mathcal D}_{k\ast l}.
 \end{equation}  see \cite[Lemma 2.8]{KLM}. A consequence of (\ref{cop}) is the following lemma.

 \begin{lemma}\label{infin} Take $k\in L^{1}_{loc}(\mathbb{R}^{+})$ such that  $0\in {supp}(k)$. Then
\item[(i)]   $ \mathcal{D}_{k^{\ast n}}\hookrightarrow \mathcal{D}_{k^{\ast m}}$ for $n\ge m \ge 1$.
\item[(ii)] $W_{k^{\ast m}}f= k^{n-m}\circ W_{k^{\ast n}}f=  W_{k^{\ast n}}(k^{n-m}\circ f)$ and if ${supp}(W_{k^{\ast n}}f)\subset I$ with $I$ an interval in $\R^+$ then
 ${supp}(W_{k^{\ast m}}f)\subset I$
 for $f\in \mathcal{D}_{k^{\ast n}}$ and $n\ge m \ge 1$.
\end{lemma}
The next definition gives the test function space will be used later to obtain new distribution spaces and corresponding distribution semigroups.
\begin{definition}\label{definf} Take $k\in L^{1}_{loc}(\mathbb{R}^{+})$ such that  $0\in {supp}(k)$. We denote by $\mathcal{D}_{k^{\ast \infty}}$ the  space defined by
$$
\mathcal{D}_{k^{\ast \infty}}:=\bigcap_{n=1}^\infty \mathcal{D}_{k^{\ast n}}.
$$
\end{definition}
It is clear that   $\mathcal{D}_{k^{\ast \infty}}$ is also a topological algebra (equipped with the $\tau$-topology) and $\mathcal{D}_{k^{\ast \infty}}\hookrightarrow \mathcal{D}_{k^{\ast n}}\hookrightarrow  \mathcal{D}_+.$ In fact, $\mathcal{D}_{k^{\ast \infty}}$ is the inverse (or projective) limit  of the family $(\mathcal{D}_{k^{\ast n}})_{n\ge 1}$. By Lemma \ref{infin}, $ W_{k^{\ast n}}:\mathcal{D}_{k^{\ast \infty}}\to \mathcal{D}_{k^{\ast \infty}}$  and  $k^{\ast n}\circ W_{k^{\ast n}}f=f $ for $f\in \mathcal{D}_{k^{\ast \infty}}$ and $n\in \N$. Note that if $f\in \mathcal{D}_{k^{\ast \infty}}$  then $f_u\in \mathcal{D}_{k^{\ast \infty}}$ for $u\ge 0$, see formula (\ref{transla}).

\medskip
\begin{example} {\rm In the case that  $\mathcal{D}_{k}= \mathcal{D}_{+}$, then  $\mathcal{D}_{k^{\ast n}}= \mathcal{D}_{+}$ for $n\in \N$ and $\mathcal{D}_{k^{\ast\infty}}= \mathcal{D}_+$. Take $z\in \C$, $e_{z}(t):=e^{zt}$ for  $t\ge 0$,
 Then  ${\mathcal
D}_{ke_z}=\{e_{z}f\,\,\vert \, f\in\Dk\}=\Dk$ and
$$W_{ke_z}f=e_zW_k(e_{-z}f), \qquad f\in \Dk;
$$
in the case that $\mathcal{D}_{k}= \mathcal{D}_{+}$, then ${\mathcal
D}_{ke_z}=\mathcal{D}_{+}$, see \cite[Proposition 2.9]{KLM}.

  (i) Recall that $\alpha >0$ and $ j_\alpha(t):={t^{\alpha-1}\over \Gamma(\alpha)}$;  the map $W_{j_\alpha}$ is the Weyl fractional derivative of order $\alpha$, $W_\alpha$ and ${\mathcal D}_{j_\alpha}={\mathcal D}_{j_\alpha^{\ast\infty}}={\mathcal D}_+$;  note that for $\alpha \in \N$,  $W_\alpha= (-1)^\alpha{d^\alpha \over dt^\alpha}$, the $\alpha$-iterate of usual derivation,  see more details for example in \cite{SKM}.

 (ii) Given $\alpha>0$ and $z\in \C$,  we have that ${\mathcal D}_{e_zj_\alpha}={\mathcal D}_{(e_zj_\alpha)^{\ast\infty}}={\mathcal D}_+$ and
 $$
 W_{e_zj_\alpha}f= e_zW_\alpha(e_{-z}f), \qquad f \in{\mathcal D}_+;
 $$
for $\alpha=1, 2$  see  explicit expressions in  \cite[Section 2]{KLM}.

 (iii) It is straightforward to check that $T'_{\chi_{(0,1)}}(f)(t)=\int_t^{t+1}f(s)ds$ for $f\in {\mathcal D}_+$,  ${\mathcal D}_{\chi_{(0,1)}}={\mathcal D}_{\chi_{(0,1)^{\ast n}}}={\mathcal D}_+$  and
$$
W_{\chi_{(0,1)}}f(t)=-\sum_{n=0}^{\infty} f'(t+n), \qquad f \in {\mathcal D}_+,\quad t\geq0.
$$

}
\end{example}
\medskip

Now let $f,g\in\mathcal{D}_{k}$. Then $f\ast g \in \mathcal{D}_{k}$ and

\medskip
\begin{eqnarray}\label{convo}
W_{k}(f\ast g)(s)=\displaystyle\int_{0}^{s}W_{k}g(r)\int_{s-r}^{s}k(t+r-s)W_{k}f(t)dtdr\\
-  \displaystyle\int_{s}^{\infty}W_{k}g(r)\int_{s}^{\infty}k(t+r-s)W_{k}f(t)dtdr\nonumber,
\end{eqnarray}
see \cite[Theorem 2.10]{KLM}.

Under some conditions on the function  $k$,  some Banach algebras under the convolution product may be considered as the next theorem shows.

\begin{theorem} \label{teoalgebras} (\cite[Theorems 3.4 and 3.5]{KLM}) Let $k\in L^1_{loc }(\R^+)$  be  a function with $0\in\textrm{supp}(k)$ and $\hbox{abs}(\vert k\vert)<\infty$.
 Then the formula
$$\|f\|_{k, e_\beta}:=\int_{0}^{\infty}|W_{k}f(t)|e^{\beta t}dt,\qquad f\in\mathcal{D}_{k},$$
for $\beta>\max\{\hbox{abs}(\vert k\vert), 0\}$ defines an algebra norm on $\mathcal{D}_{k}$ for the convolution product $\ast$. We denote by $\mathcal{T}_{k}(e_\beta)$ the Banach space obtained as the completion of $\mathcal{D}_{k}$ in the norm $\|\cdot\|_{k,e_\beta}$, and then we have $\mathcal{T}_{k}(e_\beta) \hookrightarrow L^{1}(\mathbb{R}^{+})$.
\end{theorem}

\medskip

Note that in the three examples below, the space ${\mathcal D}_k = {\mathcal D}_+$. However, as the following result shows, there are functions $k$ such that ${\mathcal D}_k \varsubsetneq  {\mathcal D}_+$ and ${\mathcal D}_{k^{\ast\infty}}=\{0\}$.

\begin{theorem} \label{estricto} Take $k\in L^{1}_{loc}(\mathbb{R}^{+})$ such that $0\in\textrm{supp}(k)$ and  $0\le \hbox{abs}(\vert k\vert)<\infty$. If $\widehat k(\lambda_0)=0$ for some $\Re\lambda_0 >\hbox{abs}(\vert k\vert)$ then ${\mathcal D}_k \varsubsetneq  {\mathcal D}_+$ and ${\mathcal D}_{k^{\ast\infty}}=\{0\}$.
\end{theorem}

\begin{proof} We suppose that ${\mathcal D}_k =  {\mathcal D}_+$ and $\widehat k(\lambda_0)=0$ for some $\Re\lambda_0 >\hbox{abs}(\vert k\vert)$. Take $\beta \in \R$  such that $\hbox{abs}(\vert k\vert)<\beta<\Re \lambda_0$,
There exists $f_n\subset \mathcal{D}_+$ such that
 \begin{equation}\label{ecc}\int_{0}^{\infty}|W_{k}f_n(t)-e_{-\lambda_0}( t)|e^{\beta t}dt\to 0 , \qquad n\to \infty,\end{equation}
As a consequence of  Theorem \ref{teoalgebras}, we obtain that $f_n \not \to 0$ in $L^1(\R^+)$. On the other hand, by \cite[Theorem 2.5 (ii)]{KLM} and (\ref{ecc}), we get that
 $$\int_{0}^{\infty}|f_n(t)-k\circ e_{-\lambda_0}(t)|e^{\beta t}dt\to 0 , \qquad n\to \infty.$$
By Lemma \ref{almo}  $k\circ e_{-\lambda_0 }=0 $, and then  $f_n\to 0 $ in $L^1(\R^+)$. We  conclude that ${\mathcal D}_k \varsubsetneq  {\mathcal D}_+$.

Now take $f\in {\mathcal D}_{k^{\ast\infty}}$. Then there exists a sequence $(g_n)\subset {\mathcal D}_+$, such that
$f= k^{\ast n}\ast g_n$ for $n\ge 1$. Then $\widehat{f}( \lambda_0)= \left(\widehat{k}(\lambda_0)\right)^n \widehat{g_n}(\lambda_0)=0$ for any $n\ge 1$. We conclude that $\widehat{f}$ has a zero on $\lambda_0$ of order  $n$  at least for any $n\ge 1$. We conclude that $f=0$.
\end{proof}

\begin{example} \label{baumer} {\rm The following example was presented in \cite[Section 5]{[Ba]} and appeared later in other references in connection to convoluted semigroups (see \cite[Example 6.1]{[KP]} and  \cite[Example 2.8.1]{ko}). Let
$$
K(\lambda):={1\over \lambda^2}\prod_{n=0}^\infty{n^2-\lambda\over n^2+\lambda}, \qquad \Re \lambda >0.
$$
Then there exists  a continuous and exponentially bounded function ${\kappa}$ in $[0,\infty)$ such that $\displaystyle\widehat \kappa= 1/{{K}}$.  Moreover, $0\in\textrm{supp}(\kappa)$ and we apply Theorem \ref{estricto} to conclude that ${\mathcal D}_{\kappa} \varsubsetneq  {\mathcal D}_+$.}

\end{example}
We note that for the cases $k(t)=j_\alpha(t), \, t>0$ (corresponding to local integrated semigroups) and $\widehat{k}(\lambda)=\displaystyle
 = \prod_{j=1}^{\infty}\left(1+{lz\over j^{1\over a}}\right)^{-1}
   $ (where $l>0,$ $0<a<1$ and  which are considered in \cite{C}) and will be presented in Section 6), we have $\widehat k(\lambda)\ne 0$ for all $\Re\lambda >\hbox{abs}(\vert k\vert).$

\section{Local convoluted semigroups}
\setcounter{theorem}{0}

The  definition of global $k$-convoluted semigroups was introduced by the
first time by I. Cioranescu \cite{C} and subsequently developed in \cite{CL} (see also \cite{[KMV]} and the monographs \cite{ko} and \cite{[MF]}). We will consider the following definition of local $k$-convoluted semigroup as appears in \cite[Definition 2.1]{[KP]}.

\begin{definition}Let $0<\tau \le \infty$, $ k\in L^1_{loc}([0,\tau))$ and $A$ be a closed operator.
$(S_k(t))_{t\in [0,\tau)}\subset \B(X)$ a strongly continuous
 operator family. The family $(S_k(t))_{t\in [0,\tau)}$ is a { local $k$-convoluted semigroup}
(or { local $k$-semigroup} in short)
 generated by $A$  if  $S_k(t)A\subset AS_k(t),$  $\int_0^tS_k(s)xds\in D(A)$ for
${t\in [0,\tau)}$ and $x\in X$ and
\begin{equation}\label{integral}
A\int_0^tS_k(t)xdt=S_k(t)x-\int_0^t k(s)dsx, \quad x\in X,
\end{equation}
for $t\in [0,\tau)$; in this case the operator $A$ is called the {
generator} of $(S_k(t))_{{t\in [0,\tau)}}$. We say that $(S_k(t))_{{t\in [0,\tau)}}$
is
{\sl non degenerate} if $S(t)x=0$ for all $ 0\le t<\tau$ implies $x=0$.
\end{definition}

Alternatively, in relation to Problem (\ref{convsem-0}), we note that when the problem is well posed in the sense that for every $x\in X,$
there exists a unique solution $\displaystyle v\in C^1([0,\, \tau), X)\cap C([0,\, \tau), D(A)),$ we set $S(t)x=v^\prime(t),\, 0\le t<\tau,\, x\in X.$ It follows from the Closed Graph Theorem that $S(t)\in \mathcal{B}(X),\, 0\le t<\tau.$ Clearly, $t\mapsto S(t)$ is strongly continuous from $[0,\,\tau)$ to $\mathcal{B}(X).$ The local convoluted semigroup defined in this manner is necessarily non degenerate, due to the uniqueness assumption.

It is easy to prove that if  $A$  generates a $k$-convoluted semigroup $(S_k(t))_{t\in [0,\tau)}$,
then $S_k(0)=0$ and $S_k(t)x\in \overline{D(A)}$ for ${t\in [0,\tau)}$ and
$x\in X$. See more details, for example in \cite{[KMV]} and  \cite{[KP]}.

\begin{remark}{\rm (i) For $\alpha>0$ and  $k(t)={t^{\alpha-1}\over \Gamma(\alpha)}$ for $t>0$, we  get $\alpha$-times integrated semigroups which were introduced  in \cite{[Hi]}. We follow the usual notation
$(S_\alpha(t))_{t\in [0,\tau)}$  for $\alpha$-times local integrated semigroups.

\noindent (ii) If $C\in\mathcal{B}(X)$ is an injective operator and we set $k(t)\equiv C,\, 0\le t<\tau$ then we recover the concept of local $C-$regularized semigroups. Local $C-$regularized semigroups were first studied in \cite{[TO90]}.

\noindent (iii) One condition in the definition of local convoluted semigroup, equation (\ref{integral}) may be interpreted as a Duhamel formula for the abstract Cauchy problem. More precisely, if we are interested in the (non-homogeneous) initial value problem
\begin{eqnarray*}\label{Cpvca}
\begin{cases}
\displaystyle{  u^\prime(t)}=A u(t)  +F(t),\, 0\le t<\tau\\
u(0)=x\in X,\\
 \end{cases}
\end{eqnarray*}
where $F$ is an $X-$valued function and $\tau\in (0,\,\infty]$, and $K(.)$ takes values in $\mathcal{B}(X)$ with the additional assumptions that $K(t)K(s)=K(s)K(t), t,s\in [0, \,\tau); \, AK(t)x=K(t)Ax, \, t\in [0,\,\tau), x\in D(A),$ we can consider the regularized problem
\begin{eqnarray*}\label{Cpvca}
\begin{cases}
\displaystyle{  v^\prime(t)}=A v(t)+K(t)x  +F_K(t),\, 0\le t<\tau\\
v(0)=0,\\
 \end{cases}
\end{eqnarray*}
in which $F_K(t)=(K\ast F)(t)=\int_0^t K(t-s)F(s)ds, \, 0\le t<\tau $ More details can be found in the reference  \cite{CL}. We hall be concerned only with the situation where $K(t)=\phi(t)I$ where $I$ is the identity operator on $X.$ Spectral criteria for the generation of local convoluted semigroups involving the resolvent of the generator can be found in the references \cite{C}, \cite{CL}, \cite{[KMV]}  and \cite{ko}.

\noindent (iv) Other equivalent definitions of local convoluted semigroup, using the composition property  (see Proposition \ref{equiv}) or the Laplace transform (\cite[Theorem 3.2]{[KMV]})  show   this algebraic aspect in a straightforward way.
}
\end{remark}

The next characterization of local $k$-semigroups has the advantage to offer an
algebraic character which is crucial in the development of the theory as we will see in Theorem \ref{main2}. The proof runs parallel as in the global case presented in \cite[Proposition 2.2]{[KP]}, see also \cite[Proposition 2.1.5]{ko}.

\begin{proposition}\label{equiv} Let $0<\tau\le \infty$, $k\in
L^1_{loc}([0,\tau))$,  $A$
 a closed linear operator and $(S_k(t))_{t\in[0,\tau)}$ a non-degenerate strongly
continuous operator family.
 Then $(S_k(t))_{t\in[0,\tau)}$ is a local $k$-convoluted semigroup generated by $A$ if and only
 if $S_k(0)=0$ and
\begin{equation}\label{algeb}
S_k(t)S_k(s)x= \int_t^{t+s}k(t+s-r)S_k(r)xdr-\int_0^{s}k(t+s-r)S_k(r)xdr, \qquad x\in X,
\end{equation}
for  $0\le s,t\le t+s<\tau$.
\end{proposition}

Note that if $k\in L^1_{loc}([0,\tau))$, and we define $(k_t)_{t\in[0,\tau)}\subset L^1_{loc}([0,\tau))$ by
$$
k_t(s):= k(t-s)\chi_{[0, t]}(s), \qquad s\in [0,\tau).
 $$
 By Proposition \ref{equiv} and Corollary  \ref{jojo}, we may conclude that $(k_t)_{t\in[0,\tau)}$ is a local $k$-convoluted semigroup in $L^1_{loc}([0,\tau))$.
In this paper, we only consider  local $k$-convoluted semigroups which are non-degenerate.

The next theorem is the main result in this paper and shows how a local $k$-convoluted semigroup $(S_k(t))_{t\in [0, \tau)}$ is extended to $[0,n \tau);$ in fact we get a local $k^{\ast n}$-convoluted semigroup in $[0,n \tau)$ for $n\in \N$. Note that we improve previous results (\cite[Section 2]{CL} and \cite[Theorem 2.1.1.9]{ko}): our approach is sharper that $n$-iterations of these theorems.

\begin{theorem}\label{ext} Let $n\in \N$, $0<\tau\le \infty$, $k \in L^1_{loc}([0,(n+1)\tau))$ and $(S_k(t))_{t\in[0,\tau)}$ be a local $k$-convoluted  semigroup generated by $A$. Then the family of operators $(S_{k^{\ast (n+1)}}(t))_{t\in [0, (n+1)\kappa]}$ defined by
$$
S_{k^{\ast (n+1)}}(t)x = \int_0^tk(t-s)S_{k^{\ast n}}(s)xds, \qquad x\in X,
$$
for    $t\in [0, n\kappa]$ and
$$
 S_{k^{\ast (n+1)}}(t)x =S_{k^{\ast n}}(n\kappa)S_{k}(t-n\kappa)x +\int_0^{n\kappa}k(t-s)S_{k^{\ast n}}(s)xds+ \int_0^{t-n\kappa}k^{\ast n}(t-s)S_{k}(s)xds, $$
 for $x\in X$ and $ t\in [ n\kappa, (n+1)\kappa]$ is a local $k^{\ast (n+1)}$-semigroup generated by $A$ for any $\kappa<\tau$.  Then we conclude that $A$ generates a local $k^{\ast (n+1)}$-semigroup $(S_{k^{\ast (n+1)}}(t))_{t\in [0, (n+1)\tau)}$.
\end{theorem}

\begin{proof} Note that the family of operators $(S_{k^{\ast (n+1)}}(t))_{t\in [0, (n+1)\kappa]}$ is strongly continuous. It is known that  $(S_{k^{\ast (n+1)}}(t))_{t\in[0,n\kappa]}$ is a local $k^{\ast (n+1)}$-semigroup generated by $A$, see for example \cite[Proposition 2.1.3]{ko} and \cite[Proposition 5.2]{KLM}. Now take    $t\in [n\kappa, (n+1)\kappa]$ and $x\in X$.  It is clear that  $S_{k^{\ast (n+1)}}(t)A\subset AS_{k^{\ast (n+1)}}(t)$ and we  show that
$
\int_0^t S_{k^{\ast (n+1)}}(r)xdr\in D(A).$ Since
$$
\int_0^t S_{k^{\ast (n+1)}}(r)xdr=\int_0^{n\kappa} S_{k^{\ast (n+1)}}(r)xdr+ \int_{n\kappa}^t S_{k^{\ast (n+1)}}(r)xdr,
$$
we check that  $\int_{n\kappa}^t S_{k^{\ast (n+1)}}(r)xdr\in D(A)$, i.e,
\begin{eqnarray*}
&\quad&\int_{n\kappa}^tS_{k^{\ast n}}(n\kappa)S_{k}(r-n\kappa)x dr\cr
&\quad&\qquad \qquad+\int_{n\kappa}^t\left(\int_0^{n\kappa}k(r-s)S_{k^{\ast n}}(s)xds+ \int_0^{r-n\kappa}k^{\ast n}(r-s)S_{k}(s)xds\right)dr\in D(A).
\end{eqnarray*}
As $(S_k(t))_{t\in[0,\tau)}$ is a local $k$-convoluted  semigroup generated by $A$, we get that $$S_{k^{\ast n}}(n\kappa)\int_{n\kappa}^t S_{k}(r-n\kappa)xdr\in D(A).$$ Now we prove that $\int_{n\kappa}^t\int_0^{n\kappa}k(r-s)S_{k^{\ast n}}(s)xdsdr\in D(A)$. We apply the Fubini theorem, change the variable $u=r-s$, to get that
\begin{eqnarray*}
&\,&\int_{n\kappa}^t\int_0^{n\kappa}k(r-s)S_{k^{\ast n}}(s)xdsdr= \int_0^{n\kappa}S_{k^{\ast n}}(s)x\int_{n\kappa}^tk(r-s)drds\cr
&\,&\int_0^{n\kappa}S_{k^{\ast n}}(s)x\int_{n\kappa-s}^{t-s}k(u)duds=\int_{0}^{t}k(u)\int_{\max\{n\kappa-u, 0\}}^{\min\{t-u, n\kappa\}}S_{k^{\ast n}}(s)xdsdu\in D(A).\cr
\end{eqnarray*}
In a similar way, it is shown that
$$
\int_{n\kappa}^t\int_0^{r-n\kappa}k^{\ast n}(r-s)S_{k}(s)xdsdr=\int_{n\kappa}^tk^{\ast n}(u)\int_0^{t-u}S_{k}(s)xdsdu \in D(A).
$$

To finish the proof we prove the equality (\ref{integral}) for $t\in [n\kappa, (n+1)\kappa]$ and $x\in X$. Note that
\begin{eqnarray*}
A\int_0^tS_{k^{\ast (n+1)}}(s)xds&=&S_{k^{\ast (n+1)}}(n\kappa)x-\int_0^{n\kappa} {k^{\ast (n+1)}}(s)dsx+ A\int_{n\kappa}^tS_{k^{\ast (n+1)}}(s)xds.\cr
\end{eqnarray*}
We apply  the Fubini theorem  to get that
\begin{eqnarray} \label{split}
\int_{n\kappa}^tS_{k^{\ast (n+1)}}(s)xds&=& S_{k^{\ast n}}(n\kappa)\int_0^{t-n\kappa}S_{k}(u)xdu+\int_0^{n\kappa}S_{k^{\ast n}}(r)x\int_{n\kappa}^t
k(s-r)dsdr\cr &\quad&\qquad
+\int_0^{t-n\kappa}S_{k}(r)x\int_{r+n\kappa}^t
k^{\ast n}(s-r)dsdr.
\end{eqnarray}
We apply the operator $A$ to the first summand to get that
$$
 S_{k^{\ast n}}(n\kappa)A\int_0^{t-n\kappa}S_{k}(u)xdu= S_{k^{\ast n}}(n\kappa)S_{k}(t-n\kappa)x-S_{k^{\ast n}}(n\kappa)x\int_0^{t-n\kappa}k(u)du.
$$

In the second summand of (\ref{split}) we write
$\int_{n\kappa}^t
k(s-r)ds= \int_{0}^{t-r}
k(u)du- \int_0^{n\kappa-r}
k(u)du$.  We apply the operator $A$ and the Fubini theorem to obtain that
\begin{eqnarray*}
&\,&A\int_0^{n\kappa}S_{k^{\ast n}}(r)x\int_{0}^{t-r}
k(u)dudr=A\int_0^tk(u) \int_0^{\min\{n\kappa, t-u\}}S_{k^{\ast n}}(r)xdrdu\cr
&\,&= \left(S_{k^{\ast n}}(n\kappa)x-\int_0^{n\kappa}k^{\ast n}(y)dyx\right)\int_0^{t-n\kappa}k(u)du\cr
&\,&\qquad + \int_{t-n\kappa}^tk(u)\left(S_{k^{\ast n}}(t-u)x-\int_0^{t-u}k^{\ast n}(y)dyx\right)du\cr
&\,&= S_{k^{\ast n}}(n\kappa)x \int_0^{t-n\kappa}k(u)du+\int_{0}^{n\kappa}k(t-r)S_{k^{\ast n}}(r)xdr \cr
&\,&\qquad -\left(\int_0^{n\kappa}k^{\ast n}(y)dy\right)\left(\int_0^{t-n\kappa}k(u)du\right)x-\int_{0}^{n\kappa}k(t-r)\int_0^{r}k^{\ast n}(y)dydrx.
\end{eqnarray*}
Using similar ideas we also get that
\begin{eqnarray*}
A\int_0^{n\kappa}S_{k^{\ast n}}(r)x\int_0^{n\kappa-r}k(u)dudr&=&\int_0^{n\kappa}k(u)\left(S_{k^{\ast n}}(n\kappa-u)x-\int_0^{n\kappa-u}k^{\ast n}(r)drx\right)du\cr
&=&
S_{k^{\ast (n+1)}}(n\kappa)x- \int_0^{n\kappa}k^{*(n+1)}(r)drx.
\end{eqnarray*}

\noindent In the  third  summand of (\ref{split}) we write $\int_{r+n\kappa}^t
k^{\ast n}(s-r)ds= \int_{r}^t- \int_{r}^{r+n\kappa}
k^{\ast n}(s-r)ds$. We apply the operator $A$ and the Fubini theorem to obtain that
\begin{eqnarray*}
&\,&A\int_0^{t-n\kappa}S_{k}(r)x\int_{r}^t
k^{\ast n}(s-r)dsdr=A\int_{0}^{t}k^{\ast n}(u)  \int_0^{\min\{t-n\kappa,t-u\}}S_{k}(r)xdrdu\cr
&\,&= S_{k}(t-nk)x\int_0^{n\kappa}k^{\ast n}(u)du+ \int_0^{t-n\kappa}k^{\ast n}(t-r)S_k(r)xdr\cr
&\,&\qquad -\left(\int_0^{n\kappa}k^{\ast n}(u)du\right)\left(\int_0^{t-n\kappa}k(u)du\right)- \int_0^{t-n\kappa}k^{\ast n}(t-r)\int_0^rk(y)dy.
\end{eqnarray*}
and finally we get
\begin{eqnarray*}
&\,&A\int_0^{t-n\kappa}S_{k}(r)x\int_{r}^{r+n\kappa}
k^{\ast n}(s-r)dsdr\cr &\,&\qquad\qquad= S_k(t-n\kappa)x\int_0^{n\kappa}k^{\ast n}(u)du-\left(\int_0^{n\kappa}k^{\ast n}(u)du \right)\left(\int_0^{t-n\kappa}k(u)du\right).
\end{eqnarray*}
We conclude that
\begin{eqnarray*}
&\,&A\int_0^{t-n\kappa}S_{k}(r)x\int_{r+n\kappa}^t
k^{\ast n}(s-r)dsdr\cr &\,&\qquad\qquad= \int_0^{t-n\kappa}k^{\ast n}(t-r)S_k(r)xdr-\int_0^{t-n\kappa}k^{\ast n}(t-r)\int_0^rk(y)dy.
\end{eqnarray*}

To finish the proof we
join together all summands  to have that
\begin{eqnarray*}
&\,&A\int_0^tS_{k^{\ast (n+1)}}(s)xds= S_{k^{\ast n}}(n\kappa)S_{k}(t-n\kappa)x+\int_{0}^{n\kappa}k(t-r)S_{k^{\ast n}}(r)xdr\cr
&\, &\qquad+ \int_0^{t-n\kappa}k^{\ast n}(t-r)S_k(r)xdr -\left(\int_0^{n\kappa}k^{\ast n}(y)dy\right)\left(\int_0^{t-n\kappa}k(u)du\right)x\cr
&\, &\qquad-\int_{0}^{n\kappa}k(t-r)\int_0^{r}k^{\ast n}(y)dydrx-\int_0^{t-n\kappa}k^{\ast n}(t-r)\int_0^rk(y)dydr\cr
&\,&= S_{k^{\ast (n+1)}}(t)x-\int_0^tk^{\ast(n+1)}(s)xds,
\end{eqnarray*}
where  we have used the Lemma \ref{lemma}. This proves the claim.
\end{proof}

In fact, the expression of the $(S_{k^{\ast (n+1)}}(t))_{t\in [0, (n+1)\kappa]}$ is not unique as shown next result. Both proof are similar to the proof of Theorem \ref{ext} and are left to the reader.

\begin{theorem}\label{ext2} Let $n\ge 2$, $0<\tau\le \infty$, $k \in L^1_{loc}([0,n\tau))$ and $(S_k(t))_{t\in[0,\tau)}$ is a local $k$-convoluted  semigroup generated by $A$. Then the family of operators $(S_{k^{\ast n}}(t))_{t\in [0, n\kappa]}$ defined  in Theorem \ref{ext} verify that
$$
S_{k^{\ast n}}(t)x = \int_0^tk^{*(n-j)}(t-s)S_{k^{\ast j}}(s)xds, \qquad x\in X,
$$
for    $t\in [0, j\kappa]$ and
\begin{eqnarray*}
 S_{k^{\ast n}}(t)x &=&S_{k^{\ast j}}(n\kappa)S_{k^{\ast(n-j)}}(t-j\kappa)x \cr
 &\quad& \qquad +\int_0^{j\kappa}k^{\ast(n-j)}(t-s)S_{k^{\ast j}}(s)xds+ \int_0^{t-j\kappa}k^{\ast j}(t-s)S_{k^{\ast(n-j)}}(s)xds, \end{eqnarray*}
 for $x\in X$, $1\le j\le n-1$ and $ t\in [ j\kappa, n\kappa]$  and  any $\kappa<\tau$.

\end{theorem}

\begin{corollary}Let  $0<\tau\le \infty$, $k \in L^1_{loc}([0,2\tau))$ and $(S_k(t))_{t\in[0,\tau)}$ be a local $k$-convoluted  semigroup. Then the family of operators $(S_{k\ast k}(t))_{t\in [0, 2\kappa]}$ defined in Theorem \ref{ext} verifies that

$$
 S_{k\ast k}(t+s)x=
 S_k(t)S_k(s)x+\left( \int_0^t+ \int_0^{s} \right)k(t+s-u)S_{k}(u)xdu
$$
   for $t, s\in[0,\kappa)$ and $x\in X.$
\end{corollary}

The next theorem  was given    in \cite[Theorem 2]{M1} in the case $n=1$.

\begin{corollary}\label{inte} Let $n\in \N$, $0<\tau\le \infty$, $(S_\alpha(t))_{t\in[0,\tau)}$ is a local $\alpha$-times integrated  semigroup generated by $A$. Then the family of operators $(S_{(n+1)\alpha}(t))_{t\in [0, (n+1)\kappa]}$ defined by
$$
S_{(n+1)\alpha}(t)x = \int_0^t{(t-s)^{\alpha-1}\over \Gamma(\alpha)}S_{n\alpha}(s)xds, \qquad x\in X,
$$
for    $t\in [0, n\kappa]$ and
$$
 S_{(n+1)\alpha}(t)x =S_{n\alpha}(n\kappa)S_{\alpha}(t-n\kappa)x +\int_0^{n\kappa}{(t-s)^{\alpha-1}\over \Gamma(\alpha)}S_{n\alpha}(s)xds+ \int_0^{t-n\kappa}{(t-s)^{n\alpha-1}\over \Gamma(n\alpha)}S_{\alpha}(s)xds, $$
 for $x\in X$ and $ t\in [ n\kappa, (n+1)\kappa]$ is a local $\alpha$-times integrated  semigroup generated by $A$ for any $\kappa<\tau$.  Then we conclude that $A$ generates a local $\alpha$-times integrated  semigroup generated by $A$,  $(S_{ (n+1)\alpha}(t))_{t\in [0, (n+1)\tau)}$.

\end{corollary}

\begin{remark} {\rm In the case $\alpha \in \N$,  and $ t\in [ n\kappa, (n+1)\kappa]$, note that
$$
\int_0^{n\kappa}{(t-s)^{\alpha-1}\over \Gamma(\alpha)}S_{n\alpha}(s)xds=\sum_{j=0}^{\alpha-1}{(t-n\kappa)^{j}\over j!}S_{(n+1)\alpha-j}(n\kappa)
$$
and
$$
\int_0^{t-n\kappa}{(t-s)^{n\alpha-1}\over \Gamma(n\alpha)}S_{\alpha}(s)xds=\sum_{j=0}^{n\alpha-1}{(n\kappa)^{j}\over j!}S_{(n+1)\alpha-j}(t-n\kappa)
$$
 for $x\in X$ and  we recover the extension given in \cite[Theorem
4.1]{[AEK]} for the case $n=1$. To finish this section, we mention that other extension result (for scalar $\alpha$-times semigroups in this case) may be found in \cite[Lemma 4.4]{[Ku]},
$$
I^n_t\ast I^n_s= I^{2n}_{s+t}-\sum_{j=0}^{n-1}\left({s^{j}\over j!}I^{2n-j}_t-{t^{j}\over j!}I^{2n-j}_s\right),\qquad t,s>0,
$$
where $I^n_t(r):={(t-r)^n\over n!}\chi_{[0,t]}(r)$, $r\in \R^+$ and $n\in \N\cup\{0\}$.}
\end{remark}

\section{Algebra homomorphisms and local convoluted semigroups}
\setcounter{theorem}{0}

As the following theorem shows, local $k$-convoluted semigroups induce algebra homomorphisms from   certain spaces of test functions $\mathcal{D}_{k^{\ast\infty}}$. Note that the extension theorem (Theorem \ref{ext}) is necessary to define the algebra homomorphisms from functions defined on $\R^+$. The space $\mathcal{D}_{k^{\ast\infty}}$ is introduced in Definition \ref{definf}.

\begin{theorem} {\label{main2}} Let $k\in L^1_{loc}(\mathbb{R}^+)$ with $0\in {supp}(k)$,
and $(S_k(t))_{t \in [0,\tau]}$ a non-degenerate local  $k$-convoluted semigroup generated by $A$.  We define the map ${\mathcal G}_k: \mathcal{D}_{k^{\ast \infty}}\to
{\mathcal B}(X)$ by
$$
{\mathcal G}_k(f)x:=\int_0^{n\tau} W_{k^{\ast n}} f(t)S_{k^{\ast n}}(t)xdt, \qquad x\in X, f\in \mathcal{D}_{k^{\ast \infty}},
$$
  where $supp(f)\subset [0,n\tau]$ and  $(S_{k^{\ast n}}(t))_{t \in [0,n\tau]}$ is defined in Theorem \ref{ext} for some $n\in \N$. Then the following properties hold.

\begin{itemize}
\item[(i)] The map ${\mathcal G}_k$ is well defined, linear and bounded.
\item[(ii)] The map ${\mathcal G}_k(f\ast g)={\mathcal G}_k(f){\mathcal G}_k(g) $  for $f, g\in \mathcal{D}_{k^{\ast \infty}}$.
\item[(iii)] ${\mathcal G}_k(f)x\in D(A)$ and $A {\mathcal G}_k(f)x= -{\mathcal G}_k(f')x-f(0)x$ for any $f\in \mathcal{D}_{k^{\ast \infty}}$ and $x\in X$.
\end{itemize}
\end{theorem}
\begin{proof} Take $f\in \mathcal{D}_{k^{\ast \infty}}$ and $supp(f)\subset [0, n\tau]$ for some $n\in \N$. First, we prove that ${\mathcal G}_k$ is well defined. Let $m\ge n$, $k^{\ast m}= k^{\ast n}\ast k^{\ast (m-n)}$, and   $k^{\ast(m-n)}\circ W_{k^{\ast m}}f =W_{k^{\ast n}}f  $ by Lemma \ref{infin} (ii). Now we apply the Lemma \ref{equivalencia} to conclude $supp(W_{k^{\ast m}}f)\subset [0, n\tau]$. By Theorem \ref{ext2} and the Fubini theorem, we get that
\begin{eqnarray*}
&\,&\int_0^{m\tau} W_{k^{\ast m}} f(t)S_{k^{\ast m}}(t)xdt=
\int_0^{n\tau} W_{k^{\ast m}} f(t)(k^{\ast(m-n)}\ast S_{k^{\ast n}})(t)xdt\cr
&\,&=\int_0^{n\tau} k^{\ast(m-n)}\circ W_{  k^{\ast m}} f(t) S_{k^{\ast n}}(t)xdt= \int_0^{n\tau} W_{k^{\ast n}} f(t)S_{k^{\ast n}}(t)xdt
\end{eqnarray*}
for $x\in X$.

It is direct to check that ${\mathcal G}_k$ is  linear. Now take $(f_n)_{n\ge 1}\subset \mathcal{D}_{k^{\ast \infty}} $, and $f \in \mathcal{D}_{k^{\ast \infty}}$ such that $f_n\to f$. Then there exists $n \in \N$ such that ${supp}(f_n),{supp}(f)\subset [0,n\tau]$. Note that the map $t \mapsto  S_{k^{\ast n}}(t)x$, $[0,n\tau]\to X$,  is continuous and
$$
\Vert {\mathcal G}_k(f_n)x-{\mathcal G}_k(f)x\Vert\le  \int_0^{n\tau} \vert W_{k^{\ast n}} f_n(t)-W_{k^{\ast n}} f(t)\vert \, \Vert S_{k^{\ast n}}(t)x\Vert dt\le C_x \int_0^{n\tau} \vert W_{k^{\ast n}} (f_n- f)(t)\vert  dt
$$
for $x\in X$. Now consider the operator $T'_{k^{\ast n}}:L^1[0,n\tau]\to L^1[0,n\tau]$, $f\mapsto T'_{k^{\ast n}}(f)= k^{\ast n}\circ f$  given in Section 3. By the open mapping theorem, $T'_{k^{\ast n}}$ is open; we conclude that $ W_{k^{\ast n}}f_n\to  W_{k^{\ast n}}f$ in $L^1[0,n\tau]$, the map ${\mathcal G}_k$ is bounded and the part (i) is proved.

Take $f, g \in \mathcal{D}_{k^{\ast \infty}}$, i.e. $ f, g \in \mathcal{D}_{k^{\ast n}}$ and then $f\ast g\in \mathcal{D}_{k^{\ast n}}$ (see \cite[Theorem 2.10]{KLM}) for $n \ge 1$. Then
   $f\ast g\in \mathcal{D}_{k^{\ast \infty}}$. Now we show that ${\mathcal G}_k(f\ast g )={\mathcal G}_k(f)
{\mathcal G}_k(g)$. Take $n\in \N$ such that $supp(f), supp(g)\subset [0,n \tau]$ and by Lemma \ref{equivalencia}, $supp(W_{k^{\ast 2n}}f), supp(W_{k^{\ast 2n}}g)\subset [0,n \tau]$. Then $supp(f\ast g)\subset [0,2n \tau]$ and $supp(W_{k^{\ast 2n}}(f\ast g))\subset [0,2n \tau]$.  By (\ref{convo}) we have that
\begin{eqnarray*}
{\mathcal G}_k(f\ast g)x&=&\int_0^{2n\tau} W_{k^{\ast 2n}}(f\ast
g)(t)S_{{k^{\ast 2n}}}(t)xdt\\
&=&\int_0^{2n\tau}\int_0^{t}W_{k^{\ast 2n}}g(r)\int_{t-r}^{t}k^{\ast 2n}(s+r-t)
W_{k^{\ast 2n}}f(s) dsdrS_{k^{\ast 2n}}(t)xdt\\
&\,\,&-\int_0^{2n\tau}\int_t^{2n\tau}W_{k^{\ast 2n}}g(r)
\int_{t}^{2n\tau}{k^{\ast 2n}}(s+r-t)W_{k^{\ast 2n}}f(s)dsdrS_{{k^{\ast 2n}}}(t)xdt.
\end{eqnarray*}
By Fubini theorem, we obtain the four integral expressions
\begin{eqnarray*}
{\mathcal G}_{k}(f\ast g)x&=&\int_0^{2n\tau}W_{k^{\ast 2n}}g(r)\int_0^{r}W_{k^{\ast 2n}}f(s)
\int_{r}^{s+r}{k^{\ast 2n}}(s+r-t)S_{k^{\ast 2n}}(t)xdtdsdr \\
&\,&\qquad+\int_0^{2n\tau}W_{k^{\ast 2n}}g(r)\int_r^{2n\tau}W_{k^{\ast 2n}}f(s)
\int_{s}^{s+r}{k^{\ast 2n}}(s+r-t)S_{k^{\ast 2n}}(t)xdtdsdr \\
&\,&\qquad-\int_0^{2n\tau}W_{k^{\ast 2n}}g(r)\int_0^{r}W_{k^{\ast 2n}}f(s)
\int_{0}^{s}{k^{\ast 2n}}(s+r-t)S_{k^{\ast 2n}}(t)xdtdsdr \\
&\,&\qquad-\int_0^{2n\tau}W_{k^{\ast 2n}}g(r)\int_r^{2n\tau}W_{k^{\ast 2n}}f(s)
\int_{0}^{r}{k^{\ast 2n}}(s+r-t)S_{k^{\ast 2n}}(t)xdtdsdr \cr &=&
\int_0^{2n\tau}W_{k^{\ast 2n}}g(r)\int_0^{r}W_{k^{\ast 2n}}f(s) \left(\int_{r}^{s+r}-
\int_{0}^{s}{k^{\ast 2n}}(s+r-t)S_{k^{\ast 2n}}(t)xdt\right)dsdr
\\
&\,&+\int_0^{2n\tau}W_{k^{\ast 2n}}g(r)\int_r^{n\tau}W_{k^{\ast 2n}}f(s)
\left(\int_{s}^{s+r}-
\int_{0}^{r}{k^{\ast 2n}}(s+r-t)S_{k^{\ast 2n}}(t)xdt\right)dsdr\\
&=& \int_0^{2n\tau}W_{k^{\ast 2n}}g(r)S_{k^{\ast 2n}}(r)\int_0^{r}W^{}_{k^{\ast 2n}}f(s)
S_{k^{\ast 2n}}(s)xdsdr
\\
&\,&+\int_0^{2n\tau}W_{k^{\ast 2n}}g(r)S_{k^{\ast 2n}}(r)\int_r^{n\tau}
W_{k^{\ast 2n}}f(s)S_{k^{\ast 2n}}(s)xdsdr\\
&=&  \int_0^{2n\tau}W_{k^{\ast 2n}}g(r)S_{k^{\ast 2n}}(r)\int_0^{2n\tau}W^{}_{k^{\ast 2n}}f(s)
S_{k^{\ast 2n}}(s)xdsdr={\mathcal G}_{k}(g){\mathcal G}_{k}(f)x,
\end{eqnarray*}
where we have applied the formula (\ref{algeb}) and the part (ii) is shown.

Now we consider $f\in \mathcal{D}_{k^{\ast \infty}}$, $supp(f)\subset [0,n\tau]$ and $x\in X$. We apply the formulae (\ref{derivada}) and  (\ref{integral}) to get
\begin{eqnarray*}
A{\mathcal G}_{k}(f)x= &=&-A\int_0^{n\tau}(W_{k^{\ast n}}f)'(t)\int_0^{t}S_{k^{\ast n}}(s)xdsdt\cr=
&-&\int_0^{n\tau}W_{k^{\ast n}}f'(t)\left(S_{k^{\ast n}}(t)x-\int_0^tk^{\ast n}(s)dsx\right)dt\\
&=&-{\mathcal G}_{k^{\ast n}}(f')x-
\int_0^{n\tau}W_{k^{\ast n}}f(t){k^{\ast n}}(t)dtx=-{\mathcal G}_{k^{\ast n}}(f')x-f(0)x,
\end{eqnarray*}
%Since $A$ is a closed operator, we conclude that $A$ is the generator of ${\mathcal G}_k$
and the part (iii) is proven.\end{proof}

The previous theorem allows to show as a consequence one of main results in \cite{KLM}.

\begin{remark}\label{notes} {\rm When the operator $A$ generates a global convoluted semigroup $(S_k(t))_{t\ge 0}$, the homomorphism ${\mathcal G}_k$ is defined from ${\mathcal D}_k$ to ${\mathcal B}(X)$ by
$$
{\mathcal G}_k(f)x=\int_0^\infty W_kf(t)S_k(t)xdt, \qquad x\in X, \qquad f\in {\mathcal D}_k,
$$
see \cite[Theorem 5.5]{KLM}. Under some conditions of the boundedness of $(S_k(t))_{t\ge 0}$,  the homomorphism  ${\mathcal G}_k$  may be extended to a  bounded Banach algebra homomorphism, see \cite[Theorem 5.6]{KLM}.}

\end{remark}

Let $k, l\in L^1_{loc}([0, \tau))$ with $0\in {supp}(k)$,
and $(S_k(t))_{t \in [0,\tau)}$ a non-degenerate local  $k$-convoluted semigroup generated by $A$. Then $(l\ast S_k(t))_{t \in [0,\tau]}$ is a non-degenerate local  $k$-convoluted semigroup generated by $A$, see a similar proof in \cite[Proposition 5.2]{KLM}.
\begin{corollary}Let $k, l\in L^1_{loc}(\R^+)$ with $0\in {supp}(k)\cap{supp}(l)$,
and $(S_k(t))_{t \in [0,\tau)}$,  $(S_{k\ast l}(t))_{t \in [0,\tau)}$  non-degenerate local  $k$-convoluted semigroups generated by $A$. Then

$$
{\mathcal G}_{k\ast l}(f)=
{\mathcal G}_{k}(f), \qquad f\in \mathcal{D}_{({k\ast l})^{\ast \infty}},
$$
where ${\mathcal G}_{k}(f), {\mathcal G}_{k\ast l} $ are defined in Theorem \ref{main2}.
\end{corollary}
\begin{proof}
Take $f \in {\mathcal D}_{{(k\ast l)}^{\ast \infty}}$ and $supp(f)\subset [0,n\tau]. $ By (\ref{cop}), $f\in  {\mathcal D}_{{(k\ast l)}^{\ast \infty}}\subset {\mathcal D}_{{k}^{\ast \infty}}$ and
\begin{eqnarray*}
&\,&{\mathcal G}_{k\ast l}(f)x=\int_0^{n\tau} W_{(k\ast l)^{\ast n}}f(t)
S_{(k\ast l)^{\ast n}}(t)xdt=\int_0^{n\tau} W_{k^{\ast n}\ast l^{\ast n}}f(t)
S_{k^{\ast n}\ast l^{\ast n}}(t)xdt\cr
&\,&=\int_0^{n\tau} W_{k^{\ast n}\ast l^{\ast n}}f(t)(l^{\ast n}\ast S_{k^{\ast n}})(t)xdt=\int_0^{n\tau} (l^{\ast n} \circ W_{k^{\ast n}\ast l^{\ast n}}f)(t) S_{k^{\ast n}}(t)xdt\cr
&\,&= \int_0^{n\tau}  W_{k^{\ast n}}f(t) S_{k^{\ast n}}(t)xdt={\mathcal G}_{k}(f)x
\end{eqnarray*}
where we have applied the formula (\ref{cop}).
\end{proof}

\begin{corollary} \label{bloco}
Let $(S_\alpha(t))_{t \in [0,\tau]}$ a non-degenerate local  $\alpha$-times integrated semigroup generated by $A$.  We define the map ${\mathcal G}_\alpha: \mathcal{D}_{+}\to
{\mathcal B}(X)$ by
$$
{\mathcal G}_\alpha(f)x:=\int_0^{n\tau} W_{\alpha n} f(t)S_{\alpha n}(t)xdt, \qquad x\in X, f\in \mathcal{D}_{+},
$$
  where $supp(f)\subset [0,n\tau]$ and  $(S_{n\alpha}(t))_{t \in [0,n\tau]}$ is defined in Theorem \ref{ext} for some $n\in \N$. Then the map ${\mathcal G}_\alpha$ is well defined, linear, bounded and ${\mathcal G}_k(f\ast g)={\mathcal G}_k(f){\mathcal G}_k(g) $  for $f, g\in \mathcal{D}_{k^{\ast \infty}}$. Moreover, ${\mathcal G}_k(f)x\in D(A)$ and $$A {\mathcal G}_k(f)x= -{\mathcal G}_k(f')x-f(0)x, \qquad  f\in \mathcal{D}_{+}, \quad x\in X.
  $$
\end{corollary}
\section{Examples and applications}
\setcounter{theorem}{0}

In this section we consider different examples of  convoluted semigroups which have been presented in the literature.  Our results are applied in all these examples to illustrate its importance.

\subsection{Differential operators on $L^p(\R^N)$} \cite{AK, [ABHN], [Hi]} Let $E$ be one of the spaces $L^p(\R^n)$ ($1\le p\le \infty$), $C_0(\R^n)$, $BUC(\R^n)$, or $C_b(\R^n)$ and $A_E$ the associated operator to a differential operator and defined by Fourier multipliers, see details in \cite[Section 4]{[Hi]},  \cite[Chapter 8]{[ABHN]}. Under some conditions, the operator $A_E$ generates an $\alpha$-times integrated semigroup $(S_\alpha(t))_{t\ge 0}$  on  $E$  and $\Vert S_\alpha(t)\Vert\le C_E (1+t^\alpha)$ for some constant $C_E$ and  certain $\alpha >0$, see  \cite[Theorem  6.3]{AK} and  \cite[Theorem 4.2]{[Hi]}. In this case the map  ${\mathcal G}_\alpha: {\mathcal D}_+\to {\mathcal B}(L^p(\R^N))$ (Corollary \ref{bloco}) extends to a Banach algebra homomorphism ${\mathcal G}_\alpha: {\mathcal T}_\alpha(1+t^\alpha) \to {\mathcal B}(L^p(\R^N))$ where ${\mathcal T}_{\alpha}(1+t^\alpha)$ is the completion of
${\mathcal D}_+$ in the norm
$$
\Vert f\Vert_{\alpha,0,\alpha }:=\int_0^\infty\vert W^\alpha f(t)\vert (1+t^{\alpha})dt, \qquad f\in {\mathcal D}_+,
$$
where $W_\alpha$ is the Weyl derivation of order $\alpha$, see \cite[Proposition 4.7]{M0}. Other examples of global $\alpha$-times integrated semigroups  may be found in \cite[Theorem 5.3]{CCO} and \cite[Proposition 8.1]{KW}. Earlier results on the Schor\"odinger equation using distribution semigroups can be found in  \cite{[BE85]} where differential operators are treated using Fourier multipliers in $L^p$ spaces.

\subsection{Multiplication local integrated  semigroup in $\ell^2$}
 \cite[Example 1.2.6]{[MF]}, \cite[Example 1]{M1}. Let $\ell^2$ be the Hilbert space of all
 complex sequences $x=(x_m)_{m=1}$ such that
 $$
 \sum_{m=1}^\infty\vert x_m\vert^2<\infty,
 $$
with the Euclidean norm  $\Vert x\Vert:= \left(\sum_{m=1}^\infty\vert
x_m\vert^2\right)^{1\over 2}$. Take $T>0$ and define
$$
a_m={m\over T}+ i\left(\left({e^m\over m}\right)^2-\left({m\over
T}\right)^2\right)^{1\over 2}, \quad m\in \N,
$$
where $i^2=-1$. For any $\alpha >0$ let  $(S_\alpha (t))_{t>0}$ be
defined by
$$
S_\alpha (t)x=\left({1\over \Gamma(\alpha )} \int_0^t(t-s)^{\alpha
-1}e^{a_m s}x_m ds\right)_{m=1},
$$
for $x\in D(S_\alpha (t))$ where  $D(S_\alpha (t))=\{x\in
\ell^2\,\, ; S_\alpha (t)x\in \ell^2\}.$ Then
$(S_\alpha(t))_{t\in[0, \alpha T)}$ is a local $\alpha $-times
integrated semigroup on $\ell^2$, $(S_\alpha (t))_{t\in[0,\alpha T)}\subset {\mathcal B}(\ell^2)$, such that  $(S_\alpha (t))_{t\in[0,\alpha T)}$ cannot be extended to $t\ge \alpha
 T$, see \cite[Example 1]{M1}. We may apply  the Corollary \ref{inte}  to define $(S_{n\alpha} (t))$ for $t<n\alpha T$ and Corollary \ref{bloco} to define the map
 ${\mathcal G}_\alpha: \mathcal{D}_{+}\to
{\mathcal B}(X)$ by
$$
{\mathcal G}_\alpha(f)x:=\int_0^{n\tau} W_{\alpha n} f(t)S_{\alpha n}(t)xdt, \qquad x\in X, f\in \mathcal{D}_{+},
$$
with $supp(f)\subset [0,n\tau]$. Other examples of local integrated semigroups defined by multiplication may be found in \cite[Example 4.4 (c), (d)]{[AEK]}.

\subsection{The Laplacian on $L^2[0,\pi]$ with Dirichlet  boundary conditions} \cite[Section 3, 5]{[Ba]}, \cite[Example 6.1]{[KP]}.  The operator $-\Delta$ on $L^2[0,\pi]$ with Dirichlet or Neumann   boundary conditions generates a polynomially bounded $\kappa$-convoluted semigroups $(S_{\kappa}(t))_{t\ge 0}$, where $\Vert S_{\kappa}(t)\Vert \le C(1+t^3)$ where $C>0$; note that  $\kappa$ is given in Example \ref{baumer} and $\vert \kappa(t)\vert\le Ce^{\beta}t$ for $t\ge 0$ and some $C, \beta >0$.
By Remarks \ref{notes}, there exists an algebra homomorphism ${\mathcal G}_\kappa: {\mathcal D}_\kappa \to {\mathcal B}(L^2[0,\pi])$ such that  it extends to a Banach algebra homomorphism ${\mathcal G}_\kappa: {\mathcal T}_\kappa(e_\beta) \to {\mathcal B}(L^2[0,\pi])$,  see Theorem \ref{teoalgebras} and \cite[Theorem 6.5]{KLM}.
%where ${\mathcal T}_\kappa(e_\beta)$ is the completion of
%${\mathcal D}_\kappa$ in the norm
%$$
%\Vert f\Vert_{\kappa, e_{-\beta}}:=\int_0^\infty\vert W_\kappa f(t)\vert e^{\beta t}dt, \qquad f\in {\mathcal D}_\kappa,
%$$
%for some $\beta >0$, see  \cite[Theorem 6.5]{KLM}.

Other examples of generators of  $k$-convoluted semigroups (which  does not generate integrated semigroups) may be found in \cite[Example 6.2]{[KP]}, or \cite[Section 2.8]{ko}, where
$$
k(t)={e^{-a^2/4t}\over 2\sqrt{\pi t^3}}=\frac{1}{2\pi i}\int_{r-i\infty}^{r+i\infty}e^{\lambda
t/{a^2}-\sqrt{\lambda}}d\lambda, \qquad t>0,
$$
for some $a>0$ (where in the integral, $r$ is any positive real number). In this case, we have:
$$\widehat{k}(\lambda)=\frac{1}{a}e^{-a\sqrt{\lambda}},\,\, \Re\lambda>0.
$$

\subsection{Ultradistributions in the Gevrey classes} \cite[Section 5. Applications ]{CL}, \cite[4. Example and final comments]{C}.
Let $M_k$, $k=0,1,2,...$ be a Gevrey type sequences, i.e., a sequence of positive numbers such that $M_0=1$, which is logarithmically convex and non-quasianalytic:
$$
\qquad M_k^2\le M_{k-1}M_{k+1}, \quad \hbox{and}\quad \sum_{k=1}^\infty M_{k-1}/M_k<\infty,
$$
for example $(k!^s)$, $(k^{ks})$ and $\Gamma(1+ks)$ for $s>1$.% We define the associated function
 %$$
 %M(r)= \sup_{k}\log(r^k/M_k), \quad r>0.
 %$$
 Let $m_k= M_k/M_{k-1}$  for $k\in \N$,
$$
P(z)= \prod_{j=1}^{\infty}\left(1+{z\over m_j}\right), \qquad \Re z>0,
$$and the function $K$ defined by $\widehat{K}(z)=1/{P(z)}$.
 %
 %Insert the estimate on the kernel here
 %
 %
 The entire function $P(z)$ is called an ultradifferential polynomial. The following two operators are generators of local $K$-convoluted semigroups and appear in \cite[4. Example and final comments]{C}:
\begin{itemize}
\item[(i)] Let $X=L^2(\R)$ and $A=i{d^4\over dt^4}-{d^2\over dt^2}.$ Then $A$ generates an $K_1$-convoluted semigroup on $[0,\tau)$, where $\widehat{K}_1(z)=1/{P_1(z)}$  where $P_1(z)= \prod_{j=1}^{\infty}\left(1+{lz\over j^2}\right)$ for some $l>0$.

    In the context of ultradistribution semigroups, this example was first proposed by J. Chazarain (\cite[Remarque 6.4]{Ch71}). It was observed in \cite{[Ke97]} that one can take $A=iB$ where $B$ is the generator of a strongly continuous cosine function. Additional results are given in \cite{ko}, \cite{[KP]} including the case of Beurling ultradistributions. Other developments are discussed in \cite{ko}, including a generalization of the abstract Weierstrass formula.

\item[(ii)] Let $\omega$ be a $\sigma$-finite measure, the Lebesgue space $X=L^p(\Omega)$ ($1\le p\le \infty$) and $m:\Omega\to \C$ a measurable function. We define $Af=mf$ where $D(A)=\{f\in L^p(\Omega); mf\in L^p(\Omega)\}$ and
    $$
    \{z\in \C; \Re z \ge \alpha \vert z\vert^a+\beta\}\subset \rho(A).
    $$
   for $\alpha, \beta>0$ and $0<a<1$. For every $\tau>0$ there is $l>0$ such that $A$ generates a local $K_2$-convoluted semigroup on $[0,\tau)$ where
   $$
   P_2(z)= \prod_{j=1}^{\infty}\left(1+{lz\over j^{1\over a}}\right),
   $$
   where $\widehat{K}_2(z)=1/{P_2(z)}$.
\end{itemize}
 The following estimates are valid for the Gevrey sequences $M_j=(j!^s)$, $M_j=(j^{js})$ and $M_j=\Gamma(1+js)$ where $s>1$ is given.
 $$
 e^{(l\vert z\vert)^{a}}\le \vert P(z)\vert\le e^{(L\vert z\vert)^{a}},\,\, \Re z\ge 0,
 $$
 where $L$ is a positive constant, see e.g.  \cite[(1.1)]{C}.

We apply the Theorem \ref{ext} to conclude that $A$  generates a local $K^{\ast n}$-convoluted semigroup on $[0, n\tau)$.

\section{Appendix: k-Distribution Semigroups}

Let $X$ be  a Banach space. Vector valued algebraic distributions, i.e., linear and continuous maps from a test function space to the space of bounded linear operators, ${\mathcal B}(X)$, which satisfy an algebraic property (similar to Theorem \ref{main2} (ii)) have been studied deeply in a large numbers of papers, see \cite[Chapter 8 ]{[Fa]}. In this sense, distribution semigroups ({\it in the sense of Lions}, DS-L) were introduced by J.L. Lions in \cite{[Li60]}, see also \cite{Ch71}, \cite[Definition 7.1]{[AEK]}.  P.C. Kunstmann considered pre-distribution semigroups (or quasidistribution semigroup in the terminology of S. W. Wang) as linear and continuous maps ${\mathcal G}: \mathcal{D} \mapsto \mathcal{B}(X)$, ${\mathcal G}\in \mathcal{D}'( \mathcal{B}(X)),$ satisfying
\begin{itemize}
\item[(i)] ${\mathcal G}(\phi\ast \psi)={\mathcal G}(\phi){\mathcal G}(\psi)$ for $\phi, \psi\in  \mathcal{D},$
\item[(ii)] $\cap \{ \ker({\mathcal G}(\theta))\,\, \vert \,\, \theta \in  \mathcal{D}_0\}=\{0\}$,
\end{itemize}
see  \cite[Definition 2.1]{[Ku]} and \cite[Definition 3.3]{[Wa]}. In fact, a pre-distribution ${\mathcal G}$ can be regarded as a continuous linear map from $\mathcal{D}_+$ into
$\mathcal{B}(X)$, $ {\mathcal G}: \mathcal{D}_+ \mapsto \mathcal{B}(X)$, such that
\begin{itemize}
\item[(i)] ${\mathcal G}(\phi\ast \psi)={\mathcal G}(\phi){\mathcal G}(\psi)$ for $\phi, \psi\in  \mathcal{D}_+.$
\item[(ii)] $\cap \{ \ker({\mathcal G}(\theta))\,\, \vert \,\, \theta \in  \mathcal{D}_+\}=\{0\}$.
\end{itemize}
see \cite[Remark 3.4]{[Wa]}. The differences between quasi-distribution and distribution semigroups in the sense of Lions may be found in \cite[Remark 3.13]{[Ku]}.

Classes of Distribution semigroups on $(0,\infty)$ (in short DS on $(0,\infty)$) were considered in \cite[Definition 1]{[Ku1]}. The subspace $ \mathcal{D}'_+( \mathcal{B}(X))\subset \mathcal{D}'( \mathcal{B}(X)) $  is formed of the elements supported in $[0,\infty)$. A distribution semigroup on $(0,\infty)$, ${\mathcal G}\in \mathcal{D}'_+( \mathcal{B}(X)),$  is a continuous linear map from $\mathcal{D}$ into
$\mathcal{B}(X)$  $ {\mathcal G}: \mathcal{D} \mapsto \mathcal{B}(X)$, supported in $[0,\infty)$, such that
\begin{itemize}
\item[(i)] ${\mathcal G}(\phi\ast \psi)={\mathcal G}(\phi){\mathcal G}(\psi)$ for $\phi, \psi\in  \mathcal{D}_0$.
\item[(ii)] $\cap \{ \ker({\mathcal G}(\theta))\,\, \vert \,\, \theta \in  \mathcal{D}_0\}=\{0\}$.
\end{itemize}
Distribution semigroups on $(0,\infty)$ and quasi-distribution semigroups are not the same concept, see \cite[Remark 4]{[Ku1]}.
However, a particular class of  distribution semigroups on $(0,\infty)$ (simply called distribution semigroups, see \cite[Definition 2]{[Ku1]}) may be identified with  quasi-distribution semigroups, \cite[Theorem 1]{[Ku1]}.

Keeping in mind these definitions and Theorem \ref{main2}, we  introduce the concept of $k$-distribution semigroup.

\begin{definition} \label{kqsd}{\rm  Let $k\in L^{1}_{loc}(\mathbb{R}^{+})$ such that $0\in\textrm{supp}(k)$.  We say that a  linear and continuous map $ {\mathcal G}_k: \mathcal{D}_{k^{\ast \infty}} \mapsto \mathcal{B}(X)$ is a $k$-distribution semigroup, in short $k$-DS, if  it satisfies the following conditions.
\begin{itemize}
\item[(i)] ${\mathcal G}_k(\phi\ast \psi)={\mathcal G}_k(\phi){\mathcal G}_k(\psi)$ for $\phi, \psi\in  \mathcal{D}_{k^{\ast \infty}}.$
\item[(ii)] $\cap \{ \ker({\mathcal G}_k(\theta))\,\, \vert \,\, \theta \in  \mathcal{D}_{k^{\ast \infty}}\}=\{0\}$.

\end{itemize}

In the case that  $ {\mathcal G}_k$ is a $k$-distribution semigroup on $(0,\infty)$, then $\mathcal{D}_{k^{\ast \infty}}\not=\{0\}$ by (ii).}

\end{definition}

\begin{remark}{\rm
 Let ${\mathcal G}: \mathcal{D}\to \mathcal{B}(X)$ be a distribution semigroup (or quasi-distribution semigroup in the sense of  Wang). Then  ${\mathcal G}\circ \Lambda$   is a $k$-distribution semigroup for any $k\in L^{1}_{loc}(\mathbb{R}^{+})$ such that  $0\in\textrm{supp}(k)$ and $\mathcal{D}_{k^{\ast \infty}}= \mathcal{D}_+$; in particular  ${\mathcal G}\circ \Lambda$ is a $j_\alpha-DS$ for any $\alpha >0$ (see definition of $\Lambda$ and $j_\alpha$ in Section \ref{seccion}).}
\end{remark}

%\begin{corollary}

%Let $k\in L^{1}_{loc}(\mathbb{R}^{+})$ such that $0\in\textrm{supp}(k)$, $\mathcal{D}_{k^{\ast \infty}}= \mathcal{D}_+$ and  ${\mathcal G}_k: \mathcal{D}_+ \mapsto \mathcal{B}(X)$ a $k$-pDS. Then ${\mathcal G}_k\circ \mathcal {K}$ is a pre-distribution.

%\end{corollary}

%\begin{proof} Take $\phi, \psi \in {\mathcal D}$. Then $\mathcal{K}(\phi \ast \psi)=\phi_+\ast \psi_+ +\phi_-\circ \psi_+ $
%$$
%w
%$$
% Note that we have to show that $\cap \{ \ker({\mathcal G}_k(\theta))\,\, \vert \,\, \theta \in  \mathcal{D}_0\}=\{0\}.$

%\end{proof}

\medskip
For a given $k$-DS ${\mathcal G}_k$, define the operator $A'$ by

\begin{itemize}
\item[(i)] $D(A'):=\cup \{ \hbox{Im}({\mathcal G}_k(\theta))\,\, \vert \,\, \theta \in  \mathcal{D}_{k^{\ast \infty}}\}$.
\item[(ii)] $A'{\mathcal G}_k(\theta)x:=- {\mathcal G}_k(\theta')x-\theta(0)x$, for $x\in X$ and $\theta \in  \mathcal{D}_{k^{\ast \infty}}.$

\end{itemize}

\begin{proposition} The operator $A'$ is well defined and closable.

\end{proposition}

\begin{proof} Assume that ${\mathcal G}_k(\phi)x= {\mathcal G}_k(\psi)y$ for  some $x, y \in X$ and $\phi, \psi\in \mathcal{D}_{k^{\ast \infty}}.$ Now take $\theta  \in \mathcal{D}_{k^{\ast \infty}}$. Since
$$
\theta\ast \phi'(t)= \theta'\ast \phi(t)+\theta(0)\phi(t)-\phi'(0)\theta(t), \qquad t\ge 0,
$$
see, for example \cite[Proposition 3.1 (iii)]{[Wa]}, we get that
$$
{\mathcal G}_k(\theta){\mathcal G}_k(\phi')x=  {\mathcal G}_k(\theta'){\mathcal G}_k (\phi)x+\theta(0){\mathcal G}_k (\phi)x-\phi'(0){\mathcal G}_k(\theta)x
$$
and hence
$$
\left(-{\mathcal G}_k(\theta')-\theta(0)\right){\mathcal G}_k (\phi)x= {\mathcal G}_k(\theta)\left(-{\mathcal G}_k(\phi')x  -\phi'(0)x\right).
$$
Similarly,
$$
\left(-{\mathcal G}_k(\theta')-\theta(0)\right){\mathcal G}_k (\psi)y= {\mathcal G}_k(\theta)\left(-{\mathcal G}_k(\psi')y  -\psi'(0)y\right).
$$
By Definition  \ref{kqsd} (ii), we conclude that
$$
-{\mathcal G}_k(\phi')x  -\phi'(0)x=
-{\mathcal G}_k(\psi')y  -\psi'(0)y
$$
and $A'$ is well defined.

To prove that $A'$ is closable, let $(x_n)_{n\ge 1}\subset D(A')$ be such that $x_n\to 0$, $A'x_n\to y$. We write $x_n={\mathcal G}_k(\phi_n)z_n$ with $(\phi_n)_{n\ge 1}\subset\mathcal{D}_{k^{\ast \infty}} $ and $(z_n)_{n\ge 1}\subset X$. Take $\theta \in \mathcal{D}_{k^{\ast \infty}}$  and then
\begin{eqnarray*}
{\mathcal G}_k(\theta)y&=& \lim_{n\to \infty}{\mathcal G}_k(\theta)A'x_n=\lim_{n\to \infty}{\mathcal G}_k(\theta)A' {\mathcal G}_k(\phi_n)z_n\\
 &=&\lim_{n\to \infty}{\mathcal G}_k(\theta)\left(- {\mathcal G}_k(\phi'_n)z_n-\phi_n(0)z_n\right)\\
 &=&\lim_{n\to \infty}\left(-{\mathcal G}_k(\theta')-\theta(0)\right){\mathcal G}_k (\phi_n)z_n=\lim_{n\to \infty}\left(-{\mathcal G}_k(\theta')-\theta(0)\right)x_n=0.
 \end{eqnarray*}
This implies that $y=0$ by Definition  \ref{kqsd} (ii).\end{proof}

\begin{definition} The closure of $A'$, denoted by $A$, is called the generator of a  ${\mathcal G}_k$.
\end{definition}
\medskip

 Other definitions of generator of distribution semigroups are given using approximate units (see \cite[Definition 7.1]{[AEK]}) or the distribution $-\delta'_0$ (\cite[Definition 3.3]{[Ku]} and \cite[Proposition 1]{[Ku1]}).  In our case, given a  $k$-DS ${\mathcal G}_k$  and its generator $(A, D(A))$, then

\begin{eqnarray*}
D(A)&\subset& \{ x\in X \vert \hbox{ exists } y \in X \hbox{ such that } {\mathcal G}_k(\theta)y=- {\mathcal G}_k(\theta')x-\theta(0)x \hbox{ for any } \theta\in  \mathcal{D}_{k^{\ast \infty}}\};\\
Ax&=&y, \qquad x\in D(A),
\end{eqnarray*}

%\hbox{ where } ${\mathcal G}_k(\theta)y=- {\mathcal G}_k(\theta')x-\theta(0)x \hbox{ for any } \theta\in  \mathcal{D}_{k^{\ast \infty}}.$

\begin{theorem}\label{khomo} Let $k\in L^1_{loc}(\mathbb{R}^+)$ with $0\in {supp}(k)$, $(S_k(t))_{t \in [0,\tau]}$ a non-degenerate local  $k$-convoluted semigroup generated by $A$ and ${\mathcal G}_k: \mathcal{D}_{k^{\ast \infty}}\to
{\mathcal B}(X)$  the map defined in Theorem \ref{main2}. Then  ${\mathcal G}_k$ is a $k$-DS generated by $A$.
\end{theorem}

\begin{proof} The first condition in Definition \ref{kqsd} appears in Theorem \ref{main2} (ii). Take $x\in X$ such that  ${\mathcal G}_k(\phi)(x)=0$ for any $\phi \in \mathcal{D}_{k^{\ast \infty}}$. Since $\phi'\in   \mathcal{D}_{k^{\ast \infty}}$,  we apply Theorem \ref{main2} (iii) to conclude  $0=\phi(0)x$ for any  $\phi\in \mathcal{D}_{k^{\ast \infty}}$ and then $0=\phi_u(0)x$ for any $u\ge 0$. We conclude that $0=\phi(u)x$ for any $u\ge 0$ and $\phi \in \mathcal{D}_{k^{\ast \infty}}$. Then $x=0$ and the condition Definition \ref{kqsd} (ii) holds.

Note that $A{\mathcal G}_k(\theta)x:=- {\mathcal G}_k(\theta')x-\theta(0)x$, for $x\in X$ and $\theta \in  \mathcal{D}_{k^{\ast \infty}}$, see Theorem \ref{main2} (iii). As $A$ is a closed operator, we conclude that $A$ is the generator of ${\mathcal G}_k$.
\end{proof}

A straightforward consequence of Theorem  \ref{khomo} is the next corollary.

\begin{corollary}\label{khomo2} Let  $(S_\alpha(t))_{t \in [0,\tau]}$ a non-degenerate local  $\alpha$-times integrated semigroup generated by $A$ and ${\mathcal G}_\alpha: \mathcal{D}_{+}\to
{\mathcal B}(X)$  the map defined in Corollary \ref{bloco}. Then  ${\mathcal G}_\alpha$ is a $j_\alpha$-DS generated by $A$ with $\alpha >0$.
\end{corollary}

When the generator $A$ is densely defined in Corollary \ref{khomo2}, it is equivalent that $A$ generates a $n$-times integrated semigroup for some $n\in \N$  and $A$ is the generator of a distribution semigroup in the sense of Lions, see \cite[Theorem 7.2, Corollary 7.3]{[AEK]}.

To consider the test-function space $\mathcal{D}_{k^{\ast \infty}}$ may be given for a wide class of vector-valued distributions such that different distribution semigroups fall into the scope of this approach.
%This appendix may be understood as a first step in a future line of research.


\begin{thebibliography}{999}
\bibitem{AK} W. Arendt, H. Kellermann, Integrated solutions of Volterra integrodifferential equations and applications, in:  Volterra Integrodifferential Equations in Banach spaces and Applications, Trento, 1987, in: Pitman Res. Notes Math. Ser., vol. 190, Longman sci. Tech., Harlow, 1989, pp. 21-51.


\bibitem{[AEK]} W. Arendt, O. El-Mennaoui  and V. Keyantuo,  Local integrated semigroups: Evolution with jumps of regularity. {\it J. Math. Anal. Appl.} {\bf 186} (1994), 572-595.



\bibitem{[ABHN]}  W. Arendt, C. Batty, M. Hieber, F. Neubrander.{
"Vector-valued Laplace Transforms and Cauchy Problems''.}  Monographs
in Mathematics. vol.{\bf 96}. Birkh\"auser, Basel, 2001.

\bibitem{[Ba]} B. B\"aumer,
Approximate solutions to the abstract Cauchy problems, in: Evolution Equations and their Application in Physical and Life Sciences, Lecture Notes in Pure and Applied Mathematics Vol {\bf 215}(Marcel Dekker, New York, 2001), 33--41.

\bibitem{[BE85]} M. Balabane, H. A.   Emamirad,
{$L^p$ estimates for Schr\"{o}dinger evolution equations.}
{\it Trans. Amer. Math. Soc.} {\bf 292}(1985),
357-373.

\bibitem{[vCa85]}J. A. van Casteren,
"Generators of strongly continuous semigroups".
Research Notes in Mathematics, {\bf 115}. Boston-London-Melbourne: Pitman Advanced Publishing Program,   1985.

\bibitem{Ch71} J. Chazarain, Probl\`emes de Cauchy abstraits et applications \`a quelques probl\`emes mixtes,  {\it J. Funct. Anal.} {\bf 7}
(1971), 386-446.

\bibitem{C} I. Cior${\breve{\hbox{a}}}$nescu, Local convoluted semigroups, in: Evolution Equations (Baton
Rouge, LA, 1992), 107--122, Dekker, New York, 1995.


\bibitem{CL} I. Cior${\breve{\hbox{a}}}$nescu, G. Lumer, Probl\`{e}mes d'\'{e}volution r\'{e}gularis\'{e}s par un noyau
g\'{e}n\'{e}ral $K(t)$. Formule de Duhamel, prolongements, th\'{e}or\`{e}mes de g\'{e}n\'{e}ration, {\it C.
R. Acad. Sci. Paris S\'{e}r}. I Math. {\bf 319} (1995), 1273--1278.



\bibitem{CCO} G. Carron, T. Coulhon and E.-L. Ouhabaz, Gaussian estimates and $L^p$-boundedness of Riesz means, {\it J. Evol. Equ.} {\bf 2} (2002) 299-317.


\bibitem{[Du]} J. M. C. Duhamel, M\'{e}moire sur la m\'{e}thode g\'{e}n\'{e}rale relative au mouvement de la chaleur dans
les corps solides plong\'{e}s dans les milieux dont la temp\'erature varie avec le temps. {\it J. Ec. Polyt.
Paris} {\bf 14}, Cah. 22, 20 (1833).


\bibitem{DL} R.M. Dubois and G. Lumer, Formule de Duhamel abstraite, {\it Arch. Math.} {\bf 43} (1984) 49-56.



\bibitem{[Fa]}
 H.O. Fattorini, "The  Cauchy Problem", Addison-Wesley Publishing Co., Reading, Mass., 1983.


%\bibitem{[Es84]}  J. Esterle,  Mittag-Leffler methods in the theory of Banach algebras and a new approach to Michael's problem. Proceedings of the conference on Banach algebras and several complex variables, 107–129, Contemp. Math., 32, Amer. Math. Soc., Providence, RI, 1984.


\bibitem{[Hi]}  M. Hieber: {Integrated semigroups and differential operators on $L^p(\R^N)$}, Math. Ann. \textbf{291} (1991), 1-16.

\bibitem{KW} C. Kaiser and L. Weis, Perturbation theorems for $\alpha$-times integrated semigroups, {\it Arch. Math.} {\bf 81} (2003) 215-228.

\bibitem{[Ke97]} V. Keyantuo, Integrated semigroups and related partial differential equations, {\it J. Math. Anal. Appl.} {\bf 212}(1997) 135-153.

\bibitem{[KMV]} V. Keyantuo, C. M\"{u}ller and P. Vieten,  The Hille-Yosida theorem
for local convoluted semigroups. {\it Proc. Edinb. Math. Soc.} {\bf 46} (2003), 357-372.


\bibitem{KLM} V. Keyantuo, C. Lizama, P. Miana, Algebra homomorphisms defined via semigroups and cosine functions, {\it J. Funct. Anal.} {\bf 257} (2009) 3454-3487.


%\bibitem{KP2} M. Kosti\'{c}, S. Pilipovi\'{c},  Convoluted $C-$operator families and abstract Cauchy problems, {\it Kragujevac J. Math.} {\bf 30} (2007), 201-210.


     \bibitem{ko} M. Kosti\'c, {"Generalized semigroups and cosine functions"}. Mathematical Institute, Belgrade, 2011.




\bibitem{[KP]} M. Kosti\'c and S. Pilipovi\'c,  Global convoluted semigroups. {\it
Math. Nachr.} {\bf 15} (2007), 1727-1743.



\bibitem{[Ku]} P. C. Kunstmann,  Distribution semigroups and abstract Cauchy Problem. {\it Trans. Amer. Math. Soc} {\bf 351}
(1999), 837--856.

\bibitem{[Ku1]} P. C. Kunstmann, M. Mijatovic and S. Pilipovic, Classes of  distribution semigroups. {\it Studia Math.} {\bf 187}
(2008), 37--58.


\bibitem{[KS]} C.C. Kuo and S.Y. Shaw, On $\alpha$-times integrated $C$-semigroups and the abstract Cauchy problem. {\it Studia Math.} {\bf 142}
(2000), 201--217.


\bibitem{LS} Y-C. Li and S-Y Shaw, On local $\alpha$-times integrated $C$-semigroup, {\it Abst. Appl. Anal.} {\bf 207} (2007) 1-18, doi  10.1155/2007/34890.

\bibitem{[Li60]} J. L. Lions, Les semi-groupes distributions, {\it Portugalia Math.} {\bf 19} (1960) 141-164.

\bibitem{[LM73]} J. L. Lions and E. Magenes, "Non-homogeneous boundary value problems and applications". Vol. III.,  Die Grundlehren der mathematischen Wissenschaften, Band 183. Springer-Verlag, New York-Heidelberg, 1973.



\bibitem{[MF]}
Melnikova, I., and A. Filinkov: "Abstract Cauchy problems: three
approaches", Chapman-Hall/CRC, New York, 2001.



\bibitem{M0} P.J. Miana, $\alpha$-Times integrated semigroups and fractional derivation, {\it Forum Math.} {\bf 14}
(2002), 23-46.


\bibitem{M1} P.J. Miana, Local and global solutions of well-posed
integrated Cauchy problems. {\it Studia Math.} {\bf 187}
(2008), 219-232.




\bibitem{MP} P. J. Miana and V. Poblete, Sharp extensions for convoluted solutions of wave equations, {\it Preprint}, 2013.



\bibitem
{SKM} S. G. Samko, A. A. Kilbas and O. I. Marichev: \emph{"Fractional Integrals and Derivatives. Theory and Applications".}
 Gordon-Breach, New York (1993).

\bibitem{[Se]} R. T. Seeley,  Extension of $C^\infty$ functions defined in a half space. {\it Proc. Amer. Math. Soc.} {\bf 15}
(1964), 625-626.

\bibitem{[TO90]} N. Tanaka, N. Okazawa,   Local C-semigroups and local integrated semigroups. {\it Proc. London Math. Soc. } {\bf 61} (1990), 63-90.

\bibitem{Sa} S. Umarov,  On fractional Duhamel's principle and its applications, {\it J. Differential Equations} {\bf 252} (2012), 5217-5234.




\bibitem{[Wa]} S. W. Wang,  Quasi-distribution semigroups and integrated semigroups. {\it J. Funct. Anal.} {\bf 146}
(1997), 352-381.


\bibitem{[WG]} S. W. Wang and M.C. Gao, Automatic extensions of local regularized semigroups and local regularized cosine funtions. {\it Proc. Amer. Math. Soc.} {\bf 127}(6)
(1999), 1651-1663.

\end{thebibliography}
\end{document}